\documentclass{amsart}
\usepackage{amsmath, amsthm, amsfonts, amssymb, amscd, stmaryrd}

\numberwithin{equation}{section}

\theoremstyle{plain}
\newtheorem{theorem}{Theorem}[section]

\newtheorem{lemma}[theorem]{Lemma}

\newtheorem{corollary}[theorem]{Corollary}

\theoremstyle{definition}

\newcommand{\sig}{\sigma}
\newcommand{\rig}{\rightarrow}
\newcommand{\Rig}{\Rightarrow}

\newcommand{\sub}{\subseteq}

\newcommand{\bs}{\backslash}

\newcommand{\lb}{\langle}
\newcommand{\rb}{\rangle}
\newcommand{\lres}{\backslash}
\newcommand{\rres}{/}
\newcommand{\di}{\Diamond}
\newcommand{\vp}{\varphi}
\newcommand{\sameas}{\leftrightharpoons}

\DeclareMathOperator{\dom}{dom}
\DeclareMathOperator{\ar}{ar}
\DeclareMathOperator{\righto}{right}
\DeclareMathOperator{\lefto}{left}
\DeclareMathOperator{\up}{up}
\DeclareMathOperator{\down}{down}
\DeclareMathOperator{\Op}{Op}
\DeclareMathOperator{\Var}{Var}

\begin{document}

\title[Universal theory of BRDG]{Complexity of the universal theory of
  bounded residuated distributive lattice-ordered groupoids}
  
\author[Shkatov]{Dmitry Shkatov} 
\address{School of Computer Science and Applied Mathematics \\
University of the Witwatersrand, Johannesburg \\
Private Bag 3 \\
Wits 2050 \\
South Africa}
\email{shkatov@gmail.com} 

\author[Van Alten]{C.J.\ Van Alten}
\address{School of Computer Science and Applied Mathematics \\
University of the Witwatersrand, Johannesburg \\
Private Bag 3 \\
Wits 2050 \\
South Africa} 
\email{clint.vanalten@wits.ac.za}

\thanks{Pre-final version of the paper published in \textit{Algebra
    Universalis}, 80(3), article 36, 2019, DOI
  https://doi.org/10.1007/s00012-019-0609-1.}

\subjclass{06F99, 06D99, 03G25}

\keywords{Universal theory, Complexity, Bounded residuated distributive lattice-ordered groupoid, 
Bounded distributive lattice with operator, Partial algebra}

\begin{abstract}
  We prove that the universal theory and the quasi-equational
  \linebreak theory of bounded residuated distributive lattice-ordered
  groupoids are both \textsf{EXPTIME}-complete.  Similar results are
  proven for bounded distributive lattices with a unary or binary
  operator and for some special classes of bounded residuated
  distributive lattice-ordered groupoids.
\end{abstract}

\maketitle

\section{Introduction}

By a {\em bounded residuated distributive lattice-ordered groupoid}, or $brdg$, for short, we mean
a bounded distributive lattice with binary operations $\circ$, $\lres$ and $\rres$ satisfying
the following property:
\[
  \mbox{$x \circ y \leqslant z \; \leftrightharpoons \; y \leqslant x \lres z\; \leftrightharpoons \; x \leqslant z \rres y$.}
\]
These algebras fall within the general framework of `residuated lattices', which are studied as algebraic semantics
for various substructural logics.
The class of $brdg$'s, in particular, is the algebraic semantics for Distributive Nonassociative Full Lambek Calculus
with Bounds. 
This logic and the corresponding class of $brdg$'s, as well as the unbounded versions thereof, 
have been investigated in connection with
categorial grammars and context-free languages by Buszkowski and Farulewski \cite{BF09},
who prove decidability of the universal theory of the class of $brdg$'s via the finite embeddability property -- see 
\cite{Far08} and \cite{BF09}. 
A \textsf{coNEXPTIME} upper bound for the universal theory of $brdg$'s is obtained
by Hanikov\'{a} and Hor\v{c}\'{i}k in \cite{HH14}.
In the present article, we improve upon the result from \cite{HH14} and obtain \textsf{EXPTIME}-completeness
of the universal theory of $brdg$'s.

The \textsf{EXPTIME} upper bound is obtained by considering the complementary problem to the universal theory,
namely, satisfiability of quantifier-free first-order formulas, to which
the methods of \cite{VA13} are applied (see also \cite{SVA}).
A quantifier-free first-order formula is satisfiable in a $brdg$ if, and only if, it is satisfiable in 
a `partial $brdg$' with cardinality not greater than the size of the formula.
By a partial $brdg$ we mean an algebraic structure
with partially defined operations that is a partial substructure of a (full) $brdg$.
We obtain here a characterization of partial $brdg$'s that allows us to describe an exponential-time algorithm
for identifying partial $brdg$'s up to a given size.
For a given formula, we then test for satisfiability in each partial $brdg$ whose size does not exceed 
the size of the formula, which can be done in time exponential in the size of the formula.

In order to obtain the \textsf{EXPTIME} lower bound for $brdg$'s, we consider the class of {\em bounded distributive lattices with operators},
or $bdo$'s, for short, by which we mean bounded distributive lattices with
unary operation $\di$ that distributes over joins and satisfies $\di 0 = 0$.
The lower bound is obtained by reduction from a two person corridor tiling problem 
that is known to be \textsf{EXPTIME}-hard \cite{Chlebus86}.
Starting with an instance of the two person corridor tiling problem, say $T$, following \cite{Chlebus86},
we construct a formula $\phi_T$ in the language of a modal logic with the universal modality (see~\cite{GP92})
such that $T$ is a `good' instance if, and only if, $\phi_T$ is satisfiable in a Kripke model.
We then construct a quantifier-free first-order formula $\varphi_T$ in the language of $bdo$'s
and show that $\varphi_T$ is satisfiable in a $bdo$ if, and only if, 
$\phi_T$ is satisfiable in a Kripke model,
establishing the \textsf{EXPTIME} lower bound for satisfiability in $bdo$'s.
Lastly, we reduce satisfiability in $bdo$'s to satisfiability in $brdg$'s, proving \textsf{EXPTIME}-completeness for
the latter decision problem.
As a consequence, we also obtained \textsf{EXPTIME}-completeness for satisfiability in $bdo$'s.
The \textsf{EXPTIME}-completeness of the universal theory of these classes then follows.

We also consider some subclasses of $brdg$'s that correspond to well-known extensions of the associated logic.
In particular, we consider $brdg$'s whose operation $\circ$ satisfies any combination of the following properties: 
commutative, decreasing, square-increasing and unital,
corresponding to the structural rules of exchange, weakening, contraction and having a truth constant.
The method used here for obtaining the \textsf{EXPTIME} upper bound for $brdg$'s is extended
to all these cases.
The  \textsf{EXPTIME} lower bound proof, however, is extended to the commutative case only.

The Nonassociative Lambek Calculus was introduced in \cite{Lam61} as the logic of sentence structure, or categorical grammar, 
see, e.g., \cite{VBM97}, \cite{Moo88}, \cite{Mor94}.
Models for this logic are `residuated ordered groupoids', that is, posets with operations 
$\circ$, $\lres$ and $\rres$ satisfying the above property.
When the poset is a lattice, such structures are algebraic models for the
Nonassociative {\em Full} Lambek Calculus.
Such algebraic structures have been considered in a different context, namely, as residuated lattices
(although associativity of $\circ$ is usually assumed).
For a recent overiew of the field, see, e.g., \cite{GJKO}.

We mention some decidability and complexity results for classes of algebras closely related to $brdg$'s.
Firstly, consider the case where the operation $\circ$ is associative and commutative.
This class has an undecidable universal theory, which follows from the undecidability
of the word problem for this class \cite{Urq84}.
The equational theory for this class, however,  is decidable \cite{GR04}.
If associativity is assumed, but not commutativity, the undecidability of the universal theory
is proved in \cite{Gal02}, whereas the equational theory is proved decidable in \cite{Koz09}.
In \cite{GJ17}, cut-elimination and finite model property are proved for various extensions of 
Distributive Full Lambek Calculus, proving decidability of
the equational theories of the corresponding classes of algebras.

Next, we mention cases in which the lattice is not required to be distributive.
This class of `lattice-ordered residuated groupoids' has an undecidable universal theory \cite{Chv15},
whereas its equational theory is in \textsf{PSPACE} \cite{BF09}.
If associativity of $\circ$ is assumed, the undecidability of the universal theory is proved in \cite{JT02}. 
%and the equational theory is in \textsf{PSPACE} (ref???).
If the associative, commutative and unital properties are assumed, undecidability of the universal theory follows
from \cite{LMSS92} (see \cite{BV02}), and the \textsf{PSPACE}-completeness of the equational theory is proved in \cite{LMSS92}.
If the lattice operations are dropped altogether, that is, we consider residuated ordered groupoids,
the (quasi-)equational theory is in \textsf{P} \cite{Bus05}, and its associative extension
has an \textsf{NP}-complete equational theory \cite{Pen06}.

Lastly, suppose the assumption of `integrality' is added, that is, the decreasing and unital properties hold
and the greatest element is the unit for the $\circ$ operation.
In the case without the distributivity axiom, the universal theory is decidable.
This was shown in \cite{BV05}, where the corresponding class of algebras is shown to have the finite embeddability property, or {\em fep},
which implies decidability of the universal theory.
The result also holds if $\circ$ is associative and/or commutative and if the lattice operations are dropped.
In the distributive lattice case, the {\em fep} is obtained in \cite{GJ17}
for various classes of integral $brdg$'s, including also the associative and/or commutative cases.

\section{Bounded residuated distributive lattice-ordered groupoids}
\label{sec:rdl}

In this section we describe the main classes of algebras that are considered in this paper and
give some of the properties we shall require in later sections.
We also describe the relational representations for these classes that form the basis of the constructions  
used for the characterizations of partial algebras in Section~\ref{sec:partial-DRL} and for establishing the lower bound complexity
in Section~\ref{sec:rdl-lower-bound}.
The relational representations and related results are obtained from \cite{DH}.

We assume familiarity with standard universal algebraic and lattice theoretic notions, as can be found in \cite{BS81}
and \cite{DP}, for example, and we follow standard notational conventions.

The algebraic structures we consider all have an underlying bounded distributive lattice structure.
Recall that the natural order associated with a distributive lattice is defined by 
the following condition: $x \leqslant y$ iff $x \wedge y = x$.
We shall include the order relation $\leqslant$ in the type of our algebras as this will be useful for
the characterization of partial structures.
We use $0$ and $1$ to denote the least and greatest elements, respectively, of a lattice.
We recall the following notions.
Let ${\mathbb{D}} = \lb D, \wedge, \vee, 0, 1, \leqslant \rb$ be a bounded distributive lattice.
A {\em filter} of ${\mathbb{D}}$ is a subset $F \sub D$ that is upward closed,
i.e., $a \in F$ and $a \leqslant b$ imply $b \in F$, and closed under meets, i.e.,
$a, b \in F$ imply $a \wedge b \in F$.
A {\em prime filter} of ${\mathbb{D}}$ is a filter $F$ such that $1 \in F$, $0 \notin F$ and, whenever $a \vee b \in F$, either $a \in F$ or $b \in F$
(equivalently, $D \bs F$ is closed under joins).

By a {\em bounded residuated distributive lattice-ordered groupoid}, $brdg$ for short, we mean an algebra
${\mathbb{A}} = \langle A, \wedge, \vee, \circ, \lres, \rres, 0, 1, \leqslant \rangle$,
where $\langle A, \wedge, \vee, 0, 1, \leqslant \rangle$ is a bounded distributive lattice,
$\circ$, $\lres$, and $\rres$ are binary operations and ${\mathbb{A}}$ satisfies the following residuation property:
\begin{equation}  \label{Res}
 x \circ y \leqslant z \;\leftrightharpoons\; y \leqslant x \lres z
  \;\leftrightharpoons\; x \leqslant z \rres y.
\end{equation}
Notice that any $brdg$ satisfies $0 \circ x = 0 = x \circ 0$ and
\begin{equation} \label{o dist}
  x \circ ( y \vee z ) = (x \circ y) \vee (x \circ z)\quad \mbox{and}\quad
  ( y \vee z ) \circ x = (y \circ x) \vee (z \circ x),
\end{equation}
that is,  $\circ$ distributes over joins.
In addition, $\circ$ is order-preserving in both co-ordinates, $\lres$ is
order-preserving in its second co-ordinate and order-reversing in its first, while $\rres$ is 
order-preserving in its first co-ordinate and order-reversing in its second.
The following properties are also satisfied by $brdg$'s:
\begin{align}  
 & x \circ (x \lres y) \leqslant y \quad\quad\quad & (y \rres x) \circ x \leqslant y   \label{dlrg 1}  \\
 & y \leqslant x \lres ( x \circ y) & y \leqslant ( y \circ x) \rres x \label{dlrg 2} \\
 & x \leqslant   ((x \circ y) \vee z) \rres y & y \leqslant x \lres ((x \circ y) \vee z).  \label{dlrg 3}
\end{align}
The class of all $brdg$'s, which we denote by $BRDG$, is a variety as it has an equational axiomatization
given by the bounded distributive lattice identities, together with the identities in \eqref{o dist}, \eqref{dlrg 1} and \eqref{dlrg 3} \cite{GO10}.

By a $brdg$-{\em frame} we mean a structure
$\mathfrak{F} = \langle P, \leqslant, R \rangle$,
where $P$ is a nonempty set, $\leqslant$ is a partial order on $P$,
 and $R$ is a ternary relation on $P$ that satisfies
 \begin{align}
  & (\forall x, x', y, z \in P) (R (x, y, z) \;\mbox{\bf and}\; x' \leqslant x \Rig R(x', y, z))  \label{FR1} \\
  & (\forall x, y, y', z \in P) (R (x, y, z) \;\mbox{\bf and}\; y' \leqslant y \Rig R(x, y', z))  \label{FR2} \\
  & (\forall x, y, z, z' \in P) (R (x, y, z) \;\mbox{\bf and}\; z \leqslant z' \Rig R(x, y, z')).  \label{FR3}
\end{align}

Given a $brdg$ ${\mathbb{A}}$, by its associated $brdg$-frame $\mathfrak{F}_{\mathbb{A}}$,
we mean the structure obtained as follows.
Let $\mathcal P$ be the set of prime filters of (the underlying distributive lattice of) ${\mathbb{A}}$ and
$\mathcal R$ the ternary relation on $\mathcal P$ defined by:
\begin{equation} \label{R def}
\mathcal R (F,G,H) \leftrightharpoons (\forall\, a,b \in A) ( a \in F \;\mbox{\bf and}\;
b \in G \Rig a \circ b \in H).
\end{equation}
Then, set $\mathfrak{F}_{\mathbb{A}} = \lb \mathcal P, \sub, \mathcal R \rb$.
It is straightforward to check that $\mathfrak{F}_{\mathbb{A}}$ satisfies (\ref{FR1}--\ref{FR3}), hence it is a $brdg$-frame.

\begin{lemma} \cite{DH}
  \label{cl:pastbox-rel}
Let ${\mathbb{A}}$ be a $brdg$ and
$\mathfrak{F}_{\mathbb{A}} = \langle \mathcal P, \sub, \mathcal R \rangle$ its associated $brdg$-frame.
Then, for $F, G, H \in \mathcal P$, the following are equivalent:
\begin{itemize}
\item[(i)] $\mathcal R(F,G,H)$,
\item[(ii)] $(\forall\, a,b \in A) (a \in F \;\mbox{\bf and}\;  a \lres b \in G \Rig b \in H)$,
\item[(iii)] $(\forall\, a, b \in A) (b \rres a \in F \;\mbox{\bf and}\;  a \in G \Rig b \in H).$
\end{itemize}
\end{lemma}

\begin{proof}
Let $F$, $G$ and $H$ be prime filters and 
suppose that $\mathcal R(F,G,H)$.
If $a, b \in A$ with $a \in F$ and $a \lres b \in G$, then $a \circ (a \lres b) \in H$. 
By \eqref{dlrg 1}, $a \circ (a \lres b) \leqslant b$, so $b \in H$, as required for (ii).
Conversely, suppose (ii) holds and let $a, b \in A$ with  $a \in F$ and $b \in G$.
By \eqref{dlrg 2}, $b \leqslant a \lres ( a \circ b)$, so
$a \lres ( a \circ b) \in G$, and hence $a \circ b \in H$, as required for $\mathcal R(F,G,H)$.
Thus, (i) and (ii) are equivalent.
A similar proof shows that (i) and (iii) are equivalent.
\end{proof}

\begin{lemma} \cite{DH}
  \label{cl:Jonsson-Tarski-fusion}
Let ${\mathbb{A}}$ be a $brdg$ and
$\mathfrak{F}_{\mathbb{A}} = \langle \mathcal P, \sub, \mathcal R \rangle$ its associated $brdg$-frame.
Let $a, b \in A$.
\begin{itemize}
\item[(i)]  If $H \in \mathcal P$ and $a \circ b \in H$, then there exist $F, G \in \mathcal P$
 such that $a \in F$, $b \in G$, and $\mathcal R(F, G, H)$.
\item[(ii)] If $G \in \mathcal P$ and $a \lres b \notin G$, then there exist $F, H \in \mathcal P$
 such that $a \in F$, $b \notin H$, and $\mathcal R(F, G, H)$.
\item[(iii)] If $F \in \mathcal P$ and $b \rres a \notin F$, then there exist $G, H \in \mathcal P$
 such that $a \in G$ and $b \notin H$ and $\mathcal R(F,G,H)$.
\end{itemize}
\end{lemma}

Let $\mathfrak{F} = \langle P, \leqslant, R \rangle$ be a $brdg$-frame.
We construct a $brdg$ ${\mathbb{A}}_\mathfrak{F}$ associated with $\mathfrak{F}$ as follows.
First, for arbitrary $X, Y \subseteq P$, define the following sets:
\begin{align*}
  X \circ Y & =  \{ z \in P \mid (\exists\, x, y \in P)\,
                    (x \in X \;\mbox{\bf and}\; y \in Y \;\mbox{\bf and}\; R(x, y, z)) \}, \\
  X \lres Y & =  \{ y \in P \mid (\forall\, x, z \in P)\,
                  ( R(x, y, z) \;\mbox{\bf and}\; x \in X \Rig  z \in Y) \}, \\
  Y \rres X & =  \{ x \in P \mid (\forall\, y, z \in P)\,
                  ( R(x, y, z) \;\mbox{\bf and}\; y \in X \Rig  z \in Y) \}.
\end{align*}
Let $\mathfrak{U}(P)$ be the set of all upclosed subsets of $P$.
Then 
$\lb \mathfrak{U}(P), \cap, \cup, \varnothing, P, \sub \rb$ is a bounded distributive lattice.
That $\circ$, $\lres$ and $\rres$ are operations on $\mathfrak{U}(P)$
follows directly from (\ref{FR1}--\ref{FR3}), and \eqref{Res} follows from the definitions of the operations.
Thus, $\langle \mathfrak{U}(P), \cap, \cup, \circ, \lres, \rres, \varnothing, P, \sub \rangle$
is a $brdg$, which we denote by ${\mathbb{A}}_\mathfrak{F}$.

Let ${\mathbb{A}}$ be a $brdg$ and $\mathfrak{F}_{\mathbb{A}} = \lb \mathcal P, \sub, \mathcal R \rb$
its associated $brdg$-frame.
For each $a \in A$, let $\mu(a) = \{ F \in \mathcal P : a \in F \}$.
Then $\mu$ is an embedding of the algebra ${\mathbb{A}}$ into the algebra ${\mathbb{A}}_{\mathfrak{F}_{\mathbb{A}}}$ \cite{DH}.
Thus, every $brdg$ can be represented as a subalgebra of a
$brdg$ ${\mathbb{A}}_\mathfrak{F}$ constructed out of a $brdg$-frame $\mathfrak{F}$.

By a {\em bounded distributive lattice with (normal) binary operator}, $bdbo$ for short, we mean an algebra  
${\mathbb{A}} = \langle A, \wedge, \vee, \circ, 0, 1, \leqslant \rangle$, where 
$\langle A, \wedge, \vee, 0, 1, \leqslant \rangle$ is a bounded distributive lattice,
$\circ$ is a binary operation and ${\mathbb{A}}$ satisfies $x \circ 0 = 0$ $= 0 \circ x$ and \eqref{o dist}.
The class of all $bdbo$'s, which we denote by $BDBO$, is a variety.

Clearly, the $\lres, \rres$-free reduct of any $brdg$ is a $bdbo$.
Moreover, every $bdbo$ can be embedded into a $brdg$ as follows.
Let ${\mathbb{A}}$ be a $bdbo$.
We construct the frame $\mathfrak{F}_{\mathbb{A}} = \langle \mathcal P, \sub \mathcal R \rangle$
as above, where $\mathcal P$ is the set of prime filters of ${\mathbb{A}}$ and 
$\mathcal R$ is defined as in \eqref{R def}.
The conditions (\ref{FR1}--\ref{FR3}) hold in this case as well, so $\mathfrak{F}_{\mathbb{A}}$ is a $brdg$-frame.
Thus, we may construct the $brdg$ ${\mathbb{A}}_{\mathfrak{F}_{\mathbb{A}}}$ into which ${\mathbb{A}}$ embeds by the map $\mu$.

By a {\em bounded distributive lattice with (normal) operator}, $bdo$ for short, we mean an algebra  
${\mathbb{A}} = \langle A, \wedge, \vee, \Diamond, 0, 1, \leqslant \rangle$, where
$\langle A, \wedge, \vee, 0, 1, \leqslant \rangle$ is a bounded distributive lattice and
$\Diamond$ is a unary operation such that ${\mathbb{A}}$ satisfies $\Diamond\, 0 = 0$ and 
$\Diamond (x \vee y) = \Diamond\, x \vee \Diamond\, y$.
We note that $\Diamond$ is order-preserving.
The class of all $bdo$'s is a variety, which we denote by $BDO$.

By a $bdo$-{\em frame} we mean a structure $\mathfrak{F} = \lb P, \leqslant, R \rb$, 
where $P$ is a non-empty set, $\leqslant$ is a partial order on $P$ and $R$ is a binary relation on  $P$ such that 
the following holds:
\begin{equation}  \label{R}
 (\forall x, x', y \in P) (R (x,y) \;\mbox{\bf and}\; x \leqslant x' \Rig R (x', y)).
\end{equation}

Given a $bdo$ ${\mathbb{A}}$, we define the
$bdo$-frame associated with ${\mathbb{A}}$ as follows.
Let $\mathfrak{F}_{\mathbb{A}} = \lb \mathcal P, \sub, \mathcal R \rb$, 
where $\mathcal P$ is the set of prime filters of ${\mathbb{A}}$
and $\mathcal R$ is the binary relation on $\mathcal P$ defined by:
\begin{equation} \label{R2 def}
\mathcal R (F, G ) \leftrightharpoons (\forall\, a \in A) ( a \in G \Rig \Diamond a \in F).
\end{equation}
It is straightforward to check that \eqref{R} is satisfied by $\mathfrak{F}_{\mathbb{A}}$, so 
$\mathfrak{F}_{\mathbb{A}}$ is a $bdo$-frame.

\begin{lemma}  \cite{DH}
  \label{JT1 lem}
Let ${\mathbb{A}}$ be a $bdo$ and $\mathfrak{F}_{\mathbb{A}} = \lb \mathcal P, \sub, \mathcal R \rb$
its associated $bdo$-frame. 
If $F \in \mathcal P$ and $a \in A$ such that
$\Diamond a \in F$, then there exists $G \in \mathcal P$ 
such that $a \in G$ and $\mathcal R(F,G)$.
\end{lemma}

Given a $bdo$-frame $\mathfrak{F} = \langle P, \leqslant, R \rangle$, we define the
$bdo$ ${\mathbb{A}}_\mathfrak{F}$ associated with it as follows.
For an arbitrary $X \subseteq P$, define the set $\Diamond X$ as follows:
\begin{equation} 
  \Diamond\, X =  \{ x \in P \mid (\exists\, y \in P)\,
                     (y \in X \;\mbox{\bf and}\; R(x, y)) \}.   \label{DB1 def}
\end{equation}
It follows from \eqref{R} that $\Diamond$ is an operation on $\mathfrak{U}(P)$, the set of all upclosed subsets of $P$.
From the definition of $\Diamond$ it follows that  $\lb \mathfrak{U}(P), \cap, \cup, \Diamond, \varnothing, P, \sub \rb$ is a $bdo$;
we denote this algebra by ${\mathbb{A}}_\mathfrak{F}$.

Let ${\mathbb{A}}$ be a $bdo$ and $\mathfrak{F}_{\mathbb{A}} = \lb \mathcal P, \sub, \mathcal R \rb$
its associated $bdo$-frame.
For each $a \in A$, let $\mu(a) = \{ F \in \mathcal P : a \in F \}$.
Clearly, $\mu(a) \in \mathfrak{U}(\mathcal P)$.
Moreover, $\mu$ is an embedding of ${\mathbb{A}}$ into ${\mathbb{A}}_{\mathfrak{F}_{\mathbb{A}}}$ \cite{DH}.
Thus, every $bdo$ can be represented as a subalgebra of a
$bdo$ ${\mathbb{A}}_\mathfrak{F}$.

\newpage

\begin{lemma}
  \label{lem:dlo-to-rdl}
Let ${\mathbb{A}} = \langle A, \wedge, \vee, \Diamond, 0, 1, \leqslant \rangle$ be a $bdo$ and let
$\circ$ be the binary operation on $A$ defined by
$x \circ y := \Diamond (x \wedge y)$.
Then, ${\mathbb{A}}^{\circ} = \langle A, \wedge, \vee, \circ, 0, 1, \leqslant \rangle$
is a $bdbo$.
\end{lemma}

\begin{proof} 
We have $0 \circ x = \Diamond (0 \wedge x) = \Diamond\, 0 = 0$ and, similarly, $x \circ 0 = 0$.
We derive
$x \circ (y \vee z) = \Diamond (x \wedge (y \vee z)) = \Diamond (
(x \wedge y) \vee (x \wedge z) ) = \Diamond (x \wedge y ) \vee
\Diamond (x \wedge z)$ $= (x \circ y) \vee (x \circ z)$.
Similarly, we may derive $(y \vee z) \circ x = (y \circ x) \vee (z \circ x)$.
\end{proof}

We note, for later, that the operation $x \circ y := \Diamond (x \wedge y)$, defined
in the above lemma, is commutative.

\begin{lemma}
  \label{lem:rdl-to-dlo}
Let ${\mathbb{A}} = \langle A, \wedge, \vee, \circ, 0, 1, \leqslant \rangle$ be
a $bdbo$ and let $\Diamond$ be the unary operation on $A$ defined by
$\Diamond x := x \circ 1$. 
Then, ${\mathbb{A}}^\Diamond = \langle A, \wedge, \vee, \Diamond, 0, 1, \leqslant \rangle$
is a $bdo$.
\end{lemma}

\begin{proof}
Indeed,
$\Diamond (x \vee y) =  (x \vee y) \circ 1 =   (x \circ 1) \vee (y \circ 1)$ $= \Diamond\, x \vee \Diamond\, y$ and $\Diamond\, 0 = 0
\circ 1 = 0$.
\end{proof}

\section{Satisfiability and universal theories} \label{Satisfiability and Universal Theories}

We recall here the notion of satisfiability of a quantifier-free first-order formula in a structure and
its relevance to the universal theory of a class of structures.
We then recall the notions of a partial structure and satisfiability in partial structures, and how this relates to satisfiability in full structures.
For more detailed background on the partial structure approach to satisfiability, we refer the reader to \cite{VA13}. 

The languages we consider in this paper consist of a finite set containing  operation symbols, constant symbols and 
the relation symbol $\leqslant$,
together with an {\em arity} function $\ar$ that assigns to each
operation symbol $\delta$ a natural number $\ar(\delta)$.
If $\sigma$ denotes a language, then we use $\sigma_{\mathrm{op}}$ and $\sigma_{\mathrm{con}}$ to denote the sets of all 
operation and constant symbols, respectively.
The relation symbol $\leqslant$ is assumed to be binary.
We write $\sigma_=$ to denote the language $\sigma$ augmented with the equality symbol.
The set of ($\sigma$-){\em terms} over a set of variables is constructed in the usual way from constants, variables and operation symbols.

Let $\sigma$ be a language of the above type and $\mathcal K$ a class of $\sigma$-structures.
We use \textsf{and}, \textsf{or}, \textsf{not}, $\Rig$ for (meta-)logical connectives.
Let $\varphi(x_1, \dots, x_n)$ be a quantifier-free
first-order formula in the language $\sigma_=$, i.e.,
a Boolean combination of atomic formulas of the form $s = t$ or $s \leq t$, where
$s$ and $t$ are terms.  
Recall that $\varphi(x_1, \dots, x_n)$ is {\em satisfiable} in a structure ${\mathbb{A}}$ of type $\sigma$
if there exists a valuation $v: \{ x_1, \dots, x_n \} \rig A$ 
such that $\vp$ is true in ${\mathbb{A}}$ under $v$.
Then, $\varphi(x_1, \dots, x_n)$ is {\em satisfiable} in $\mathcal K$ if there exists ${\mathbb{A}} \in \mathcal K$
such that $\vp$ is satisfiable in ${\mathbb{A}}$.
We refer to the problem of deciding if a given $\varphi$ is
satisfiable in $\mathcal K$ as {\em satisfiability} in $\mathcal K$.

By a {\em universal sentence} we mean a first-order sentence of the
form \linebreak
$\Phi: = (\forall x_1, \dots, x_n) \varphi(x_1, \dots, x_n)$ where
$\vp$ is a quantifier-free first-order formula \linebreak (in the
language $\sigma_=$).  For such a sentence, we write
$\mathcal K \models \Phi$ to mean that $\vp$ is true in every
${\mathbb{A}} \in \mathcal K$ under every valuation.  By the {\em
  universal theory} of $\mathcal{K}$, we mean the set of all universal
sentences $\Phi$ for which $\mathcal K \models \Phi$ .  Observe that
$\mathcal K \not\models \Phi$ means that there exists
${\mathbb{A}} \in \mathcal K$ and valuation $v$ such that $\vp$ is not
true in ${\mathbb{A}}$ under $v$, i.e., $\mbox{\bf not\,}\varphi$ is
true in ${\mathbb{A}}$ under $v$.  Thus, $\mathcal K \models \Phi$ iff
$\mbox{\bf not\,}\varphi$ is not satisfiable in $\mathcal K$.  The
complexity of the universal theory of $\mathcal K$, therefore, yields
the complexity of satisfiability in $\mathcal K$, and vice versa.

Let $X$ and $Y$ be sets.
By a {\em partial function} from $X$ to $Y$ we mean a function
$\tau: X' \rig Y$ where $X'$ is a subset of $X$ called the {\em domain} of $\tau$
and denoted by $\dom(\tau)$.
For any $a \in X$, we say that $\tau(a)$ is {\em defined} iff $a \in \dom(\tau)$.
Let $A$ be a non-empty set and $n$ a natural number such that $n \geq 1$.
An $n$-{\em ary partial operation} on $A$ is a partial function $\tau$ from
$A^n$ to $A$; $n$ is called the {\em arity} of $\tau$.

Let $\sigma$ be a language of the type described above.
A structure ${\mathbb{B}}$ is called a {\em partial $\sigma$-structure} if it consists of the following: a nonempty set $B$,
an $\ar(\delta)$-ary partial operation $\delta^{{\mathbb{B}}}$ on $B$, 
for each $\delta \in \sigma_{\mathrm{op}}$,
a constant $c^{{\mathbb{B}}} \in B$
for each $c \in \sigma_{\mathrm{con}}$ 
and a binary relation $\leqslant^{{\mathbb{B}}}$ on $B$.

If ${\mathbb{B}}$ is a partial $\sigma$-structure then for every
term $t(x_1,\dots, x_n)$ there is an associated partial function
$t^{{\mathbb{B}}}$ from $B^n$ to $B$.
For $(a_1, \dots, a_n) \in B^n$, $t^{\mathbb{B}}(a_1, \dots, a_n)$ is defined iff 
$t$ can be evaluated in ${\mathbb{B}}$ under the valuation $v(x_i) = a_i$
(that is, if all partial operations required for the valuation of $t$ in ${\mathbb{B}}$ under $v$ are defined).

Let ${\mathbb{B}}$  be a partial $\sigma$-structure, $\varphi(x_1, \dots, x_n)$ a quantifier-free first-order $\sigma_{=}$-formula and
$v: \{ x_1, \dots, x_n \} \rig B$ a valuation.
We say that $\varphi(x_1, \dots, x_n)$ is {\em satisfied} in ${\mathbb{B}}$ under $v$ if
$t^{{\mathbb{B}}}(v(x_1), \dots, v(x_n))$ is defined for every term $t$ occurring in $\varphi$ and $\varphi$
is true in ${\mathbb{B}}$ under $v$.
We say that $\varphi$ is {\em satisfiable} in ${\mathbb{B}}$ iff $\varphi$ is satisfied  in ${\mathbb{B}}$ under some valuation.

Let ${\mathbb{A}}$ be a $\sigma$-structure and ${\mathbb{B}}$ a partial $\sigma$-structure.
We say that ${\mathbb{B}}$ is a {\em partial substructure}
of ${\mathbb{A}}$ if $B \sub A$, 
$\delta^{{\mathbb{B}}}(a_1, \dots, a_{\ar(\delta)}) = \delta^{\mathbb{A}}(a_1, \dots, a_{\ar(\delta)})$
for all $\delta \in \sigma_{\mathrm{op}}$ and $(a_1, \dots, a_{\ar(\delta)}) \in \dom(\delta^{{\mathbb{B}}})$,
$c^{{\mathbb{B}}} = c^{\mathbb{A}}$ for all $c \in \sigma_{\mathrm{con}}$ and
$\leqslant^{{\mathbb{B}}} = \, \leqslant^{\mathbb{A}} \cap \, (B^2)$.

Let $\varphi$ be a quantifier-free first-order $\sigma_{=}$-formula.
Let $|\Op(\vp)|$ be the total number of occurrences of all operation symbols in $\vp$
and let $|\Var(\vp)|$ be the number of distinct variables occurring in $\vp$.
The {\em size} of $\varphi$ is then defined as
$s(\varphi) := |\Op(\vp)| + |\Var(\vp)| + |\sigma_{\mathrm{con}}|$.

\begin{theorem} \cite{VA13}  \label{main theorem}
Let $\mathcal K$ be a class of $\sigma$-structures and 
$\varphi$ a quantifier-free first-order $\sigma_{=}$-formula.
Then $\varphi$ is satisfiable in $\mathcal K$ if, and only if, $\varphi$
is satisfiable in some partial substructure ${\mathbb{B}}$ of some member of $\mathcal K$, with $|B| \leq s(\varphi)$.
\end{theorem}

If we have an algorithm for identifying the partial substructures of members of $\mathcal K$ 
up to a given size amongst the set of all partial $\sigma$-structures,
then, by the above theorem, we have an algorithm for deciding
satisfiability in $\mathcal K$.
In addition, we obtain an upper bound for the complexity of 
satisfiability in $\mathcal K$ if the complexity of
identifying the partial substructures of members of $\mathcal K$ is known.
In that case, we also have an upper bound for
the complexity of the universal theory of $\mathcal K$.

Notice that, conversely, if we can decide satisfiability in $\mathcal K$,
then we can decide if a given partial structure is a substructure of some member of 
$\mathcal K$.
To see this, suppose one can decide satisfiability in $\mathcal K$, and
let ${\mathbb{B}}$ be a partial structure in the language of $\mathcal K$ with domain $B = \{ a_1, \dots, a_n \}$.
Construct a quantifier-free first-order formula $\varphi$ in the language of $\mathcal K$ describing ${\mathbb{B}}$ as follows: 
use $\{ x_1, \dots, x_n \}$ as variables representing the elements of ${\mathbb{B}}$ and set
$\varphi$ as the conjunction of the following literals: 
$x_i \neq x_j$ for $i \neq j$; 
for every $k$-ary function symbol $\delta$ in the language and each $i_1, \dots, i_k$, if $\delta^{\mathbb{B}}(a_{i_1}, \dots, a_{i_k}) = a_j$,
include $\delta(x_{i_1}, \dots, x_{i_k}) = x_j$;
for every $k$-ary relation symbol $R$ in the language and each $i_1, \dots, i_k$,
include $R(x_{i_1}, \dots, x_{i_k})$ if $R^{\mathbb{B}}(a_{i_1}, \dots, a_{i_k})$, and {\bf not}$R(x_{i_1}, \dots, x_{i_k})$, otherwise.
Clearly, $\varphi$ is satisfiable in ${\mathbb{B}}$ under the valuation $v(x_i) = a_i$. 
Moreover, for any ${\mathbb{A}}$ in $\mathcal K$, $\varphi$ is satisfiable in ${\mathbb{A}}$ if, and only if, 
${\mathbb{B}}$ is a partial substructure of ${\mathbb{A}}$.
Thus, $\varphi$ is satisfiable in $\mathcal K$ if, and only if, 
${\mathbb{B}}$ is a partial substructure of some member of $\mathcal K$.

\section{Partial $brdg$'s}
\label{sec:partial-DRL}

We now provide a characterization of partial $brdg$'s.
We start by giving characterizations of partial bounded (distributive) lattices, as obtained in \cite{VA13}.
Let $\sigma^{b\ell}$ be the language containing two binary operation symbols $\wedge$ and $\vee$, two constant symbols 0 and 1, 
and a binary relation symbol $\leqslant$.
Let ${\mathbb{B}} = \lb B, \wedge^{\mathbb{B}}, \vee^{\mathbb{B}}, 0^{\mathbb{B}}, 1^{\mathbb{B}}, \leqslant^{\mathbb{B}} \rb$
be a partial $\sigma^{b\ell}$-structure.
Then, ${\mathbb{B}}$ is a partial substructure of a bounded lattice, i.e., a {\em partial bounded lattice}, 
if $\leqslant^{\mathbb{B}}$ is a partial order on $B$ with
 $0^{\mathbb{B}}$ and $1^{\mathbb{B}}$ as {\em bounds}, that is, 
 they are the least and greatest elements of $B$ with respect to $\leqslant^{\mathbb{B}}$,
 and the partial operations $\wedge^{{\mathbb{B}}}$ and $\vee^{{\mathbb{B}}}$ 
are {\em compatible} with $\leqslant^{{\mathbb{B}}}$ in the following sense: for all $a, b \in B$,
\begin{itemize}
\item[$\bullet$] if $a \wedge^{\mathbb{B}} b$ is defined, then $a \wedge^{{\mathbb{B}}} b$ is the greatest lower bound of $\{ a,b \}$
with respect to $\leqslant^{\mathbb{B}}$,
\item[$\bullet$] if $a \vee^{\mathbb{B}} b$ is defined, then $a \vee^{{\mathbb{B}}} b$ is the least upper bound of $\{ a, b \}$
with respect to $\leqslant^{\mathbb{B}}$.
\end{itemize}
We may characterize partial bounded distributive lattices as follows.
Let ${\mathbb{B}}$ be a partial bounded lattice.
By a {\em prime filter} of ${\mathbb{B}}$ we mean a subset $f$ of $B$ such that, for all $a, b \in B$,
  \begin{align}
  & \mbox{$1^{\mathbb{B}} \in f$ and $0^{\mathbb{B}} \notin f$;}  \label{PF1} \\
  & \mbox{if $a \in f$ and $a \leqslant^{\mathbb{B}} b$, then $b \in f$;}  \label{PF2} \\
  & \mbox{if $a \in f$ and $b \in f$ and $a \wedge^{\mathbb{B}} b$
    is defined, then $a \wedge^{\mathbb{B}} b \in f$;}  \label{PF3}  \\
  & \mbox{if $a \notin f$ and $b \notin f$ and
    $a \vee^{\mathbb{B}} b$ is defined, then $a \vee^{\mathbb{B}} b \notin f$.}  \label{PF4}
  \end{align}
By \cite{VA13}, a partial $\sigma^{b\ell}$-structure ${\mathbb{B}}$ is a partial substructure of a bounded
distributive lattice,
i.e., {\em partial bounded distributive lattice}, if it is a partial bounded lattice and 
there exists a set $\mathcal F$ of prime filters of ${\mathbb{B}}$ such that the following holds:
\begin{itemize}
  \item[(D)] $(\forall\, a, b \in B) [ a \not\leqslant^{\mathbb{B}} b \Rig
    (\exists\, f \in \mathcal{F} ) ( a \in f \;\mbox{\bf and}\; b \notin f )]$.
\end{itemize}

We use $\sig^{brdg}$ to denote the language of $brdg$'s, that is $\sig^{brdg}$
contains binary operation symbols $\wedge$, $\vee$, $\circ$, $\lres$ and $\rres$,
constant symbols 0 and 1 and a binary relation symbol  $\leqslant$. 
By a {\em partial $brdg$} we shall mean a partial $\sig^{brdg}$-structure
that is a partial substructure of a $brdg$.

In the following theorem we characterize partial $brdg$'s.
Clearly, a partial $brdg$ must contain a partial bounded distributive lattice; as we show, the
set $\mathcal F$ of prime filters required for (D) must also form the basis for a $brdg$-frame. 

\begin{theorem}
  \label{partial dlrg thm}
Let ${\mathbb{B}} = \lb B, \wedge^{\mathbb{B}}, \vee^{\mathbb{B}}, \circ^{\mathbb{B}}, \lres^{\mathbb{B}}, \rres^{\mathbb{B}}, 
0^{\mathbb{B}}, 1^{\mathbb{B}}, \leqslant^{\mathbb{B}} \rb$
be a partial $\sig^{brdg}$-structure.  
Then, ${\mathbb{B}}$ is a partial $brdg$ if, and only if, 
its $\sigma^{b\ell}$-reduct is a partial bounded lattice
and there exists a set $\mathcal{F}$ of prime
filters of ${\mathbb{B}}$ such that {\em (D)} holds, as well as
\begin{flalign*}
\mbox{\em (M$_{\circ}$)} \quad (\forall\, h \in \mathcal{F}) (\forall\, (a,b) \in & \dom(\circ^{\mathbb{B}})) [a \circ ^{\mathbb{B}} b \in h \Rig & \\
    & (\exists f, g \in \mathcal{F}) (a \in f \;\mbox{\bf and}\; b \in g  \;\mbox{\bf and}\; R^{\mathbb{B}} (f,g,h))], &
\end{flalign*}
\begin{flalign*}
\mbox{\em (M$_{\lres}$)} \quad (\forall\, g \in \mathcal{F}) (\forall\, (a,b) \in & \dom( \lres^{\mathbb{B}})) [ a \lres^{\mathbb{B}} b \notin g \Rig & \\
    & (\exists\, f, h \in \mathcal{F}) (a \in f \;\mbox{\bf and}\; b \notin h  \;\mbox{\bf and}\; R^{\mathbb{B}} (f,g,h))], &
\end{flalign*}
\begin{flalign*}
\mbox{\em (M$_{\rres}$)} \quad (\forall\, f \in \mathcal{F}) (\forall\, (a,b) \in & \dom( \rres^{\mathbb{B}})) [ a \rres^{\mathbb{B}} b \notin f \Rig & \\
    & (\exists\, g, h \in \mathcal{F}) (a \in g \;\mbox{\bf and}\; b \notin h  \;\mbox{\bf and}\; R^{\mathbb{B}} (f,g,h))], & 
\end{flalign*}
where
\[
   \begin{array}{rcl}
     R^{\mathbb{B}} ( f, g, h ) &  \leftrightharpoons & (\forall\, (a,b) \in \dom(\circ^{\mathbb{B}}))
    ( a \in f \;\mbox{\bf and}\; b \in g \Rig a \circ^{\mathbb{B}} b \in h) \;\mbox{\bf and}\; \\
     & & (\forall\, (a,b) \in \dom(\lres^{\mathbb{B}} )) ( a \in f \;\mbox{\bf and}\; a \lres^{\mathbb{B}} b \in
            g \Rig b \in h) \;\mbox{\bf and}\; \\
     & & (\forall\, (a,b) \in \dom(\rres^{\mathbb{B}} )) ( b \rres^{\mathbb{B}} a \in f \;\mbox{\bf and}\; a  \in
            g \Rig b \in h).
   \end{array}
\]
\end{theorem}

\begin{proof}
Suppose ${\mathbb{B}}$ is a partial $brdg$; say ${\mathbb{B}}$ is a partial substructure of a
$brdg$ ${\mathbb{A}}$.
Since $\leqslant^{\mathbb{B}} \, = \, \leqslant^{\mathbb{A}} \cap \, (B^2)$, it follows that $\leqslant^{\mathbb{B}}$ is a partial order
with $0^{\mathbb{B}}$ and $1^{\mathbb{B}}$ as bounds, and
that $\wedge^{\mathbb{B}}$ and $\vee^{\mathbb{B}}$ are compatible with $\leqslant^{\mathbb{B}}$.
Thus, the $\sigma^{b\ell}$-reduct of ${\mathbb{B}}$ is a partial bounded lattice.
We show that there exists a set $\mathcal F$ of prime
filters of ${\mathbb{B}}$ that satisfies (D), (M$_\circ$), (M$_\lres$) and (M$_\rres$).  
Let
\[
  \mathcal{F} := \{ F \cap B \mid F \mbox{ is a prime filter of } {\mathbb{A}} \}.
\]
Note that each $F \cap B \in \mathcal F$ is a prime filter of ${\mathbb{B}}$ since $a \wedge^{\mathbb{B}} b = a \wedge^{\mathbb{A}} b$
and $a \vee^{\mathbb{B}} b = a \vee^{\mathbb{A}} b$ whenever $a \wedge^{\mathbb{B}} b$ and $a \vee^{\mathbb{B}} b$ are defined.
To verify (D), let $a, b \in B$ and assume that $a \not\leqslant^{\mathbb{B}} b$, hence
also $a \not\leqslant^{\mathbb{A}} b$.  
By properties of distributive lattices,
there exists a prime filter $F$ of ${\mathbb{A}}$ such that $a \in F$
and $b \notin F$.  Then, $f := F \cap B \in \mathcal F$ is a witness to the
satisfaction of (D).
  
To verify (M$_\circ$), (M$_\lres$) and (M$_\rres$), we first make the following observation.
Recall that $\mathfrak{F}_{\mathbb{A}} = \langle \mathcal P, \sub, \mathcal R \rangle$, where $\mathcal P$ is the set of
prime filters of ${\mathbb{A}}$ and $\mathcal R$ is
defined as in \eqref{R def}, is the $brdg$-frame associated with ${\mathbb{A}}$.
For $F, G, H \in \mathcal P$, let $f:= F \cap B$, $g := G \cap B$ and $h := H \cap B$.
Then,
\begin{equation}   \label{R and R}
    \mbox{if $\mathcal R(F,G,H)$, then $R^{\mathbb{B}}(f, g, h)$}.
\end{equation}
To see this, let $\mathcal R(F,G,H)$ and suppose, first, 
that $(c,d) \in \dom(\circ^{\mathbb{B}})$,  $c \in f$ and $d \in g$.
Then $c \in F$ and $d \in G$, so $c \circ^{\mathbb{A}} d \in H$.
As ${\mathbb{B}}$ is a partial substructure of
${\mathbb{A}}$,  $c \circ^{\mathbb{B}} d = c \circ^{\mathbb{A}} d \in H$.
Also, $c \circ^{\mathbb{B}} d \in B$,
hence $c \circ^{\mathbb{B}} d \in h$.
Next, suppose that $(c, d) \in \dom(\lres^{\mathbb{B}})$,  $c \in f$ and $c \lres^{\mathbb{B}} d \in g$.
Since $c \lres^{\mathbb{A}} d = c \lres^{\mathbb{B}} d$, we have $c \lres^{\mathbb{A}} d \in G$; also $c \in F$, so
$c \circ^{\mathbb{A}} (c \lres^{\mathbb{A}} d) \in H$.
By \eqref{dlrg 1}, therefore, we have $d \in H$, hence $d \in h$, as required.
The third clause of the definition of $R^{\mathbb{B}}$ follows similarly.

To see that (M$_{\circ}$) holds, assume that $h \in \mathcal{F}$, $(a, b) \in \dom( \circ^{\mathbb{B}})$ and
$a \circ^{\mathbb{B}} b \in h$.  
Then $h = H \cap B$, for some prime filter $H$ of ${\mathbb{A}}$ and
$a \circ^{\mathbb{A}} b = a \circ^{\mathbb{B}} b \in H$. 
In view of Lemma~\ref{cl:Jonsson-Tarski-fusion}(i), there exist prime filters $F$
and $G$ of ${\mathbb{A}}$ such that $a \in F$, $b \in G$, and $\mathcal R(F,G,H)$.
Thus, for $f := F \cap B$ and $g := G \cap B$ we have $a \in f$, $b \in g$ and $R^{\mathbb{B}}(f,g,h)$, by \eqref{R and R}.
That (M$_{\lres}$) and (M$_{\rres}$) hold follows in a similar way from
Lemma~\ref{cl:Jonsson-Tarski-fusion}(ii) and (iii), respectively.

Next, assume that ${\mathbb{B}}$ is a partial $\sig^{brdg}$-structure satisfying the requirements of the theorem.
We construct a $brdg$ into which ${\mathbb{B}}$ can be embedded, thereby showing that
${\mathbb{B}}$ is isomorphic to a substructure of a $brdg$.
Observe that the structure
$\mathfrak{F} = \langle \mathcal{F}, \subseteq, R^{\mathbb{B}} \rangle$ satisfies (\ref{FR1}--\ref{FR3}), hence it is a $brdg$-frame.
Let ${\mathbb{A}}_\mathfrak{F} = \langle \mathfrak{U}(\mathcal{F}), \cap, \cup, \circ, \lres, \rres,
\varnothing, \mathcal{F}, \sub \rangle$ be the $brdg$ associated with this $brdg$-frame.
For each $a \in B$, let $\mu( a ) = \{ f \in \mathcal{F} \mid a \in f \}$.
We show that $\mu$ is an embedding of ${\mathbb{B}}$ into
${\mathbb{A}}_\mathfrak{F}$.
For $a, b \in B$, if $a \leqslant^{\mathbb{B}} b$, then $\mu(a) \sub \mu(b)$, and if
$a \not\leqslant^{\mathbb{B}} b$, then, by (D), there exists $f \in \mathcal F$ with $a \in f$ and $b \notin f$,
so $\mu(a) \not\sub \mu(b)$.
This shows that $a \leqslant^{\mathbb{B}} b$ iff $\mu(a) \sub \mu(b)$, and also
that $\mu$ is one-to-one.
In addition, by (\ref{PF1}--\ref{PF4}), $\mu(0^{\mathbb{B}}) = \varnothing$, $\mu(1^{\mathbb{B}}) = \mathcal F$,
$\mu(a \wedge^{\mathbb{B}} b) = \mu(a) \cap \mu(b)$ for all $(a,b) \in \dom(\wedge^{\mathbb{B}})$, and
$\mu(a \vee^{\mathbb{B}} b) = \mu(a) \cup \mu(b)$ for all $(a,b) \in \dom(\vee^{\mathbb{B}})$.
For  $(a, b) \in \dom( \circ^{\mathbb{B}})$, we have $\mu( a \circ^{\mathbb{B}} b ) = \{ h \in \mathcal{F} \mid a \circ^{\mathbb{B}} b \in h \}$
and
\[
    \mu(a ) \circ \mu(b ) =  \{ h \in \mathcal{F} \mid (\exists\, f, g \in \mathcal{F} )\,
                                       (a \in f \;\mbox{\bf and}\; b \in g \;\mbox{\bf and}\; R^{\mathbb{B}} (f, g, h))  \}.
\]
That $\mu( a \circ^{\mathbb{B}} b ) = \mu(a ) \circ \mu(b)$ follows
directly from (M$_{\circ}$) and the definition of $R^{\mathbb{B}}$.
That $\mu(a \lres^{\mathbb{B}} b ) = \mu(a ) \lres \mu(b )$ for
$(a,b) \in \dom(\lres^{\mathbb{B}})$ and
$\mu(a \rres^{\mathbb{B}} b ) = \mu(a ) \rres \mu(b )$ for \linebreak
$(a,b) \in \dom(\rres^{\mathbb{B}})$ follows easily from
(M$_{\lres}$), (M$_{\rres}$) and the definition of $R^{\mathbb{B}}$.
\end{proof}

\section{Complexity of satisfiability in $BRDG$: Upper bound}
\label{sec:dol-complexity}

We now prove that satisfiability in $BRDG$ is in \textsf{EXPTIME}.
We describe an algorithm that, given a quantifier-free first-order $\sig^{brdg}_=$-formula
$\vp$, determines whether there exists a $brdg$ ${\mathbb{A}}$ and a
valuation $v$ such that $\vp$ is true in ${\mathbb{A}}$ under $v$, that is, if $\vp$ is satisfiable in $BRDG$.
In view of Theorem~\ref{main theorem}, $\vp$ is true in some
$brdg$ under some valuation if, and only if,
$\vp$ is true in a partial $brdg$ whose cardinality does not exceed
$s(\vp) = |\Var( \vp )| + |\Op( \vp )| + 2$, under some valuation.
The algorithm described below establishes
if such a partial $brdg$ and valuation exist for $\vp$.

We start by describing an algorithm for determining if a given (finite)
partial $\sig^{brdg}$-structure is a partial $brdg$.
Let ${\mathbb{B}}$ be a partial $\sig^{brdg}$-structure with finite $B$.
To check whether ${\mathbb{B}}$ is a partial $brdg$, using Theorem~\ref{partial dlrg thm},
we carry out the following steps:

\begin{enumerate}
\item Check that $\leqslant^{\mathbb{B}}$ is a partial order on $B$ with bounds $0^{\mathbb{B}}$ and $1^{\mathbb{B}}$
and that $\wedge^{\mathbb{B}}$ and $\vee^{\mathbb{B}}$ are
  compatible with $\leqslant^{\mathbb{B}}$. 
\item Check whether there exists a set $\mathcal{F}$ of prime filters
  of $\mathbb{B}$ that satisfies conditions (M$_{\circ}$), (M$_{\lres}$), (M$_{\rres}$) and (D).
\end{enumerate}
Step (1) can be done in time polynomial in $|B|$; if the conditions hold we proceed to step (2), otherwise the algorithm terminates
with the negative answer.
We now describe how step (2) can be carried out in time
exponential in $|B|$.
First, generate the set $\mathcal F_0$ of all prime filters of ${\mathbb{B}}$.  To
that end, we check, for every $S \subseteq B$, if it satisfies
conditions (\ref{PF1}--\ref{PF4}).  
As checking these conditions for a given $S \subseteq B$ can be
done in time polynomial in $|B|$, this step takes time 
$\mathcal{O} (r(|B|) \times 2^{|B|})$ for some polynomial $r$.

Next, for each $f \in \mathcal F_0$, we check that the following three properties hold:
\begin{align*}
 & (\forall (a,b) \in \dom(\circ^{\mathbb{B}}))[a \circ^{\mathbb{B}} b \in f \Rig  \\
 & \quad\quad\quad\quad\quad (\exists f', f'' \in \mathcal F_0)(a \in f' \;\mbox{\bf and}\; b \in f'' \;\mbox{\bf and}\; R^{\mathbb{B}}(f', f'', f))],
 \end{align*}
 \begin{align*}
 & (\forall (a,b) \in \dom(\lres^{\mathbb{B}}))[a \lres^{\mathbb{B}} b \notin f \Rig \\
 & \quad\quad\quad\quad\quad (\exists f', f'' \in \mathcal F_0)(a \in f' \;\mbox{\bf and}\; b \notin f'' \;\mbox{\bf and}\; R^{\mathbb{B}}(f', f, f''))],
 \end{align*}
 \begin{align*}
 & (\forall (a,b) \in \dom(\rres^{\mathbb{B}}))[a \rres^{\mathbb{B}} b \notin f \Rig \\
 & \quad\quad\quad\quad\quad (\exists f', f'' \in \mathcal F_0)(a \in f' \;\mbox{\bf and}\; b \notin f'' \;\mbox{\bf and}\; R^{\mathbb{B}}(f, f', f''))],
\end{align*}
where $R^{\mathbb{B}}$ is as defined in Theorem~\ref{partial dlrg thm}.
Notice that checking $R^{\mathbb{B}}(f,g,h)$ for prime filters $f,g,h$, is polynomial
in $|B|$ and, thus, checking whether the above properties hold, for a given
$f$, can be done in time $\mathcal{O} (q(|B|) \times 2^{2|B|})$, for some polynomial $q$.
If any of the above properties fails for $f$, that is, no suitable $f'$, $f''$ are found, we eliminate $f$ from $\mathcal F_0$.
Otherwise, we retain $f$ in $\mathcal F_0$
and repeat the process with the next element of $\mathcal F_0$.

Having, in the above way, traversed the entire $\mathcal F_0$,
we obtain the resultant set $\mathcal F_1$ of prime filters. 
If we had eliminated at least one $f \in \mathcal F_0$, but not all of them, 
the above procedure is then repeated on $\mathcal F_1$.
Traversing the set $\mathcal F_i$ to obtain the set $\mathcal F_{i+1}$ in this way
is repeatedly carried out until either we obtain the empty set or no
eliminations have been done on the latest pass, resulting in the
non-empty set $\mathcal F$ of prime filters of ${\mathbb{B}}$.
At each traversal, we go through 
a list of length at most $2^{|B|}$, and the entire number of
traversals does not exceed $2^{|B|}$, as at each one (except possibly the last) we
eliminate at least one element of the remaining list.
Therefore, this procedure
can be carried out in time $\mathcal{O} (q(|B|) \times 2^{ 4|B|})$, which is $\mathcal{O} (2^{ 5|B|})$.

If the empty set has been produced, the algorithm terminates
with the negative answer.
Otherwise, the set $\mathcal F$ satisfies the properties (M$_{\circ}$), (M$_{\lres}$) and (M$_{\rres}$),
and this is the largest set of prime filters of ${\mathbb{B}}$ that satisfies these properties
(notice that the ordering of the prime filters in $\mathcal F_0$ does not affect the outcome).
For such an $\mathcal F$, we proceed to check if it satisfies (D).
Note that if (D) does not hold for $\mathcal F$, then it does not hold for any subset of $\mathcal F$.
If (D) holds, then ${\mathbb{B}}$ is a partial $brdg$; if not, the negative answer is returned.
Checking (D) can be done in time $\mathcal{O} (|B|^2 \times 2^{ |B|})$,
hence checking whether ${\mathbb{B}}$ is a partial $brdg$ can be done in time
$\mathcal{O} (2^{ 5|B|})$.

We next estimate how many candidate structures ${\mathbb{B}}$ we need to check.  
As noticed above, if $\vp$ is true in a $brdg$
under some valuation, it is true in a partial $brdg$ with
cardinality not greater than $s(\vp)$.
Thus, we need to check all structures ${\mathbb{B}}$ with $|B| \leqslant s(\vp)$
and all valuations in such structures.
Each such structure can be encoded using matrices corresponding to the operations and the partial order,
each of which has size $\mathcal{O} (s(\vp)^2)$.
Each entry in the matrices can
take on $\mathcal{O} (s(\vp))$ values (which includes an `undefined' value).
Thus, the total number of structures to be checked is $\mathcal{O} ((k_1s(\vp))^{k_2s(\vp)^2})$, which
is $\mathcal{O} (2^{ks(\vp)^3})$ (where $k_1, k_2, k$ are positive constants).
Checking if $\varphi$ holds in a structure ${\mathbb{B}}$ under a given valuation
can be done in time $\mathcal O(s(\varphi))$.
Checking if $\varphi$ holds in ${\mathbb{B}}$ under some valuation
requires considering all valuations; this can be done in time $\mathcal O(s(\vp) \times |B|^{|\Var(\varphi)|})$, which is $\mathcal{O} (2^{s(\vp)^3})$.

Thus, the algorithm runs in time $\mathcal{O} (2^{ks(\vp)^3} \times
2^{ 5 s(\vp)} \times 2^{s(\vp)^3})$, that is, in time exponential in $s(\vp)$.
This establishes the following result.

\begin{theorem}
  \label{cl:BDOs-upper-bound}
Satisfiability in $BRDG$ is in {\em \textsf{EXPTIME}}.
\end{theorem}

We use Theorem~\ref{cl:BDOs-upper-bound} to obtain \textsf{EXPTIME} upper bounds for satisfiability in $BDBO$ and $BDO$.
The languages of $bdbo$'s and $bdo$'s are denoted by $\sigma^{bdbo}$ and $\sigma^{bdo}$, respectively.

\begin{lemma}  \label{vp dlbo lem}
Let $\varphi$ be a quantifier-free first-order $\sigma^{bdbo}_=$-formula.
Then $\vp$ is satisfiable in a $bdbo$ if, and only if, it is satisfiable in a $brdg$.
\end{lemma}
\begin{proof}
If $\vp$ holds in a $brdg$ ${\mathbb{A}}$ under valuation $v$, then
it also holds in the $\lres, \rres$-free reduct of ${\mathbb{A}}$, which is a $bdbo$,
under the same valuation.
Conversely, suppose $\vp$ holds in a $bdbo$ ${\mathbb{A}}$ under valuation $v$.
As shown in Section~\ref{sec:rdl}, ${\mathbb{A}}$ embeds into the $brdg$  ${\mathbb{A}}_{\mathfrak{F}_{\mathbb{A}}}$
by the map $\mu$.
Then, $\vp$ holds in ${\mathbb{A}}_{\mathfrak{F}_{\mathbb{A}}}$ under the valuation $v' = \mu \circ v$.
\end{proof}

The following result now follows immediately from Theorem~\ref{cl:BDOs-upper-bound}.

\begin{corollary}  \label{BDBO ub cor}
Satisfiability in $BDBO$ is in {\em \textsf{EXPTIME}}.
\end{corollary}

Next, we consider satisfiability in $bdo$'s.
For any quantifier-free first-order $\sigma^{bdo}_=$-formula  $\vp$, let
$\vp^{\circ}$ be the formula obtained by replacing in $\vp$ all terms
of the form $\Diamond t$ by $t \circ 1$.
Then $\vp^{\circ}$ is a quantifier-free first-order formula in the language $\sigma^{bdbo}_=$
(and also in the language $\sigma^{brdg}_=$).

\begin{lemma}  \label{vp* lem}
Let $\vp$ be a quantifier-free first-order $\sigma^{bdo}_=$-formula.
Then $\vp$ is satisfiable in a $bdo$ if and only if, $\vp^{\circ}$ is satisfiable in a $bdbo$.
\end{lemma}

\begin{proof}
Suppose that $\vp$ holds in a $bdo$ ${\mathbb{A}}$ under an assignment $v$.
Define the operation $\circ$ on ${\mathbb{A}}$ by $x \circ y := \Diamond (x \wedge y)$.
Then, in view of Lemma~\ref{lem:dlo-to-rdl}, the algebra 
${\mathbb{A}}^{\circ} = \langle A, \wedge, \vee, \circ, 0, 1 \rangle$ is a $bdbo$.
To verify that $\vp^{\circ}$ holds in ${\mathbb{A}}^{\circ}$ under $v$, it suffices to notice that
every term of the form $t \circ 1$ occuring in $\vp^{\circ}$ evaluates
to $\Diamond (t \wedge 1)$, which equals $\Diamond\, t$.

Conversely, suppose that $\vp^{\circ}$ holds in a $bdbo$ ${\mathbb{A}}$ under
an assignment $v$.
Define an operation $\Diamond$ on ${\mathbb{A}}$ by $\Diamond x := x \circ 1$.
Then, in view of Lemma~\ref{lem:rdl-to-dlo}, the algebra
${\mathbb{A}}^\Diamond = \langle A, \wedge, \vee, \Diamond, 0, 1 \rangle$ is a $bdo$.
To verify that $\vp$ holds in ${\mathbb{A}}^\Diamond$ under $v$,
it suffices to notice that every term of the form
$\Diamond\, t$ occuring in $\vp$ evaluates to $t \circ 1$.
\end{proof}

\begin{corollary}
Satisfiability in $BDO$ is in {\em \textsf{EXPTIME}}.
\end{corollary}
\begin{proof}
Observe that constructing the formula $\vp^{\circ}$ from a
quantifier-free first-order $\sigma^{bdo}_=$-formula  $\vp$ can be done in polynomial time.
Thus, the result follows from Lemma~\ref{vp* lem} and Corollary~\ref{BDBO ub cor}.
\end{proof}

\section{Complexity of satisfiability in $BRDG$: Lower bound}
\label{sec:rdl-lower-bound}

In this section, we establish the \textsf{EXPTIME} lower bound for satisfiability in $BDO$.
By embedding $bdo$'s into $bdbo$'s and $brdg$'s we obtain
similar results for the classes $BDBO$ and $BRDG$.
Consequently, we obtain \textsf{EXPTIME}-completeness for satisfiability in $BDO$, $BDBO$ and $BRDG$,
and also for the universal theories of these classes.

The lower bound for $BDO$ is established by a reduction from an \textsf{EXPTIME}-hard
two person corridor tiling problem from~\cite{Chlebus86}.
The reduction proceeds in two steps.  
Starting with an instance, say $T$, of the two person corridor tiling problem, following \cite{Chlebus86}, 
we construct a formula $\phi_T$ in the language of the logic
$L_{[\forall]}$, which is a modal logic with the universal (or, global) modality $[\forall]$ (see~\cite{GP92}).
We use the fact that $T$ is a `good' instance iff $\phi_T$ is satisfiable in a Kripke model.
The argument is similar to \cite{Chlebus86} (see also~\cite{Blackburn}, Theorem 6.52). 
Then, we construct a quantifier-free first-order $\sigma^{bdo}_=$-formula $\varphi_T$
that is satisfiable in a $bdo$ if, and only if, 
$\phi_T$ is satisfiable in a Kripke model.

We start by describing the syntax and semantics of $L_{[\forall]}$. 
The language $\sigma^{m\ell u}$ of $L_{[\forall]}$
consists of the binary operation symbols $\wedge$ and $\vee$, the
unary operation symbols $\neg$, $\Diamond$ and $[\forall]$, and the
constant symbol $\top$.
Formulas, or terms, are defined in the usual way.
We also write $\Box \vp$ for $\neg \Diamond \neg \vp$.

Satisfiability of $\sigma^{m\ell u}$-formulas is defined in terms of Kripke semantics.
For more background on Kripke semantics we refer
the reader to~\cite[Chapter 2]{Blackburn}.
A {\em Kripke frame} is a pair
$\langle W, R \rangle$, where $W$ is a non-empty set of worlds and
$R$ is a binary accessibility relation on $W$. 
A {\em Kripke model} is a
pair $\mathfrak{M} = \langle \mathfrak{F}, V \rangle$, where $\mathfrak{F}$ is a
Kripke frame and $V$ is a valuation function that assigns to every
variable a subset of $W$.  The satisfaction relation between Kripke
models $\mathfrak{M}$, worlds $w$, and $\sigma^{m\ell u}$-formulas $\phi$
is defined as follows:
\begin{itemize}
\item $\mathfrak{M}, w \models p \sameas\ w \in V(p)$, \quad for every variable $p$;
\item $\mathfrak{M}, w \models \top$ always holds;
\item $\mathfrak{M}, w \models \neg \phi_1 \sameas\ \mathfrak{M}, w \not\models \phi_1$;
\item $\mathfrak{M}, w \models \phi_1 \wedge \phi_2 \sameas\ \mathfrak{M}, w \models \phi_1 \;\mbox{and}\;
  \mathfrak{M}, w \models \phi_2$;
\item $\mathfrak{M}, w \models \phi_1 \vee \phi_2 \sameas\ \mathfrak{M}, w \models \phi_1 \;\mbox{or}\;
  \mathfrak{M}, w \models \phi_2$;
\item $\mathfrak{M}, w \models \Diamond \phi_1 \sameas\ R(w, v) \;\mbox{and}\;
  \mathfrak{M}, v \models \phi_1$, for some $v \in W$;
\item $\mathfrak{M}, w \models [\forall] \phi_1 \sameas\ \mathfrak{M}, v \models \phi_1$, for
  every $v \in W$.
\end{itemize}
We say that a $\sigma^{m\ell u}$-formula $\phi$ is {\em satisfiable} in a Kripke
model if there exists a Kripke model $\mathfrak{M}$ and a world $w$ such
that $\mathfrak{M}, w \models \phi$.  

We now describe the two person corridor tiling game from \cite{Chlebus86} on which the 
 two person corridor tiling problem is based.
Our description of the game closely
follows that in \cite[Section 6.8]{Blackburn}.
The game is played by two players, Eloise and Abelard,
who are placing tiles of types $T_1, \ldots, T_{s+1}$ into a grid, or
corridor, comprised of $n$ columns of infinite height, according to
the rules described below.  The boundaries of the grid are delimited
by tiles of a special type $T_0$. Each side of every tile $t$ is
coloured; the colors are denoted by $\lefto(t)$, $\righto(t)$,
$\up(t)$, and $\down(t)$.  If $t$ is of type $T_0$, then
$\lefto(t) = \righto(t) = \up(t) = \down(t)$.  A tile of type
$T_{s+1}$, like those of type $T_0$, is special--it is a ``winning'' tile.

At the start of a play of the game, the initial $n$ tiles are already
in place in the first row of the grid.  The players then take turns to
place tiles into the grid. Eloise goes first in every play of the
game.  The rules of placement are as follows.  The grid has to be
filled continuously, from left to right and from bottom to top.  The
tile $t$ being placed has to match the colours of the neighbouring
tiles already in place, including the boundary tile of type $T_0$, in
the sense that $\lefto(t) = \righto(t')$ for the tile $t'$ to the left
of $t$, $\down(t) = \up(t'')$ for the tile $t''$ straight below $t$, and
$\righto(t) = \lefto(t''')$ for the tile $t'''$, if it exists, to
the right of $t$ (such a tile would always be of type $T_0$).
If a winning tile of type $T_{s+1}$ is placed into the first
column of a row, Eloise wins.  
Otherwise, Abelard wins (in particular, he
wins if the play goes on indefinitely or if one of the players cannot make a move).
 
A winning strategy for a player is defined as usual.
An instance, say $T$, of this corridor tiling game consists of the set $\{ T_0, \dots, T_{s+1} \}$
of tile types, the number $n$ of columns in the corridor and the types $T_{I_1}, \dots, T_{I_n}$
of the tiles placed in the first row.

The problem of deciding, for a given instance $T$ of the two person corridor tiling game, 
if Eloise has a winning strategy in $T$ is called the two person corridor tiling problem;
it is known to be \textsf{EXPTIME}-hard \cite{Chlebus86}.

Given an instance $T$ of the two person corridor tiling game, 
we construct a $\sigma^{m\ell u}$-formula $\phi_T$ such
that $\phi_T$ is satisfiable in a Kripke model if, and only if, Eloise has a winning
strategy in $T$.
We follow closely the construction used in \cite[Theorem 6.52]{Blackburn} for
Propositional Dynamic Logic.

In constructing $\phi_T$, we use the following variables:
\begin{itemize}
\item $p_1, \ldots, p_n$, to represent the grid position into which a
  tile is to be placed in the current round of the play;
\item $c_i(T_u)$, where $i \in \{0, \ldots, n + 1\}$ and
  $u \in \{0, \ldots, s + 1 \}$, to assert that the tile placed in the
  topmost row of column $i$ is of type $T_u$;
\item $e$, to represent whose turn it is to make a move in the current
  round of the play ($e$ for Eloise; $\neg e$ for Abelard);
\item $w$, to assert that the current position is a winning one for
  Eloise, i.e., she has a winning strategy starting from the current
  position;
\item $q_1, \ldots, q_m$ for representing the number of rounds in the
  game, using a binary encoding, where $m = \lceil \log_2 (s+2)^{n+2} \rceil$ (see below).
\end{itemize}

The following $\sigma^{m\ell u}$-formula describes the initial
position of a play:
$$
Init = e \wedge p_1 \wedge c_0 (T_0) \wedge c_1 (T_{I_1}) \wedge \ldots \wedge
c_n (T_{I_n}) \wedge c_{n+1} (T_0).
$$

We now describe the rules of the game with $\sigma^{m\ell u}$-formulas.

\noindent
Tiles are placed in exactly one of the columns 1 though $n$:
  $$
  R_1 = [\forall] ((p_1 \vee \ldots \vee p_n) \wedge \bigwedge_{i = 1}^{n}
  \bigwedge_{j \ne i} (\neg p_i \vee \neg p_j)).
  $$
In every column, a tile of exactly one type has been previously
  placed:
  $$
  R_2 = [\forall] \bigwedge_{i = 0}^{n + 1} (c_i(T_0) \vee \ldots \vee
  c_i(T_{s+1})) \wedge [\forall] \bigwedge_{i = 0}^{n + 1}
  \bigwedge_{u = 0}^{s+1} \bigwedge_{v \ne u} (\neg c_i(T_u) \vee
  \neg c_i(T_v)).
  $$
Columns $0$ and $n+1$ always contain a tile of type $T_0$:
  $$
  R_3 = [\forall] (c_0(T_0) \wedge c_{n+1} (T_0)).
  $$
The tiles are always placed in successive positions:
  $$
  R_4 =  [\forall] ( ( \neg p_1 \vee \Box p_2) \wedge ( \neg p_2 \vee \Box
  p_3) \wedge \ldots \wedge ( \neg p_n \vee \Box p_1)).
  $$
In a column where no tile is being placed, nothing changes once
  a move has been made:
  $$
  R_5 = [\forall] \bigwedge_{i = 0}^{n + 1}  \bigwedge_{u = 0}^{s + 1} ( p_i
  \vee ( (\neg c_i(T_u) \vee \Box c_i(T_u)) \wedge (c_i(T_u) \vee \Box
  \neg c_i(T_u)) ) ).
  $$
Players alternate in their moves:
  $$
  R_6 = [\forall] (( \neg e \vee \Box \neg e ) \wedge ( e \vee \Box  e )).
  $$
Players only place tiles that match the ones placed to the
    left and below:
  $$
  R_7 = [\forall]  \bigwedge_{i=1}^n (( \neg (p_i \wedge c_{i-1} (T') \wedge c_i (T'')) \vee \Box
  \bigvee_{C(T', T'', T)} c_i (T))),
  $$
  where $C(T', T'', T)$ holds if, and only if, $\righto(T') = \lefto(T)$
  and $\up(T'' ) = \down(T )$.
In column $n$, players only place tiles that match the
    boundary tile to the right:
  $$
  R_8 = [\forall] ( \neg p_n \vee \Box \bigvee_{\righto(T) = \lefto(T_0)} c_n (T)).
  $$
At his turn, Abelard can place any tile permitted by the rules
  of the game:
\begin{align*}
  R_9 & = [\forall] \bigwedge_{i = 1}^{n-1} ( \neg (\neg e \wedge p_i \wedge c_{i-1}
  (T'') \wedge c_i (T')) \vee \bigwedge_{C(T', T'', T)} \Diamond c_i
  (T)) \wedge \\
 & [\forall] ( \neg (\neg e \wedge p_n \wedge c_{n-1}
  (T'') \wedge c_n (T')) \vee \bigwedge_{C(T', T'', T), \righto(T) = \lefto(T_0)} \Diamond c_n
  (T)).
\end{align*}

We next say, using $\sigma^{m\ell u}$-formulas, that Eloise has a winning strategy in a play.
The initial position is a winning position for Eloise. 
At all other positions, one of the following holds:
  \begin{itemize}
  \item a winning tile of type $T_{n+1}$ has been placed in column 1;
  \item if Eloise is to move at the current position, then there exists a move
    to a winning position for Eloise;
  \item if Abelard is to move at the current position, then he can
    make a move and all his moves result in a winning position for
    Eloise.
  \end{itemize}
 Thus, we define
  $$
  Win = w \wedge [\forall] ( \neg w \vee c_1( T_{s+1}) \vee (e \wedge \Diamond w )
  \vee (\neg e \wedge \Diamond \top \wedge \Box w ) ).
  $$
The play does not run forever.  
First, notice that, since the number of tile types is finite, the play runs forever if, and only
if, a row in the tiling is repeated in the course of the play.
Second, observe that such a repetition is bound to occur once the
play has gone on for $N = (s+2)^{n+2}$ rounds.  
We can represent all the numbers $0$ through $N$ in binary using propositional variables
  $q_1, \ldots, q_m$, where $m = \lceil \log_2 N \rceil$. 
Let
  $$
  I_0 =  q_1 \vee ( \Box q_1 \wedge \bigwedge_{j=2}^{m} ( ( \neg q_j \vee
  \Box q_j) \wedge ( q_j \vee \Box \neg q_j) ) ),
  $$
  $$
  I_1^i = \neg ( \neg q_{i+1} \wedge \bigwedge_{j=1}^{i} q_j) \vee (
  \Box q_{i+1} \wedge \bigwedge_{j=1}^{i} \Box \neg q_j \wedge
  \bigwedge_{k=i+2}^{m} (( \neg q_k \vee \Box q_k) \wedge ( q_k \vee
  \Box \neg q_k) ) ).
  $$
  Now, let
  $$
  F = \neg q_m \wedge \ldots \wedge \neg q_1 \wedge [\forall] ( I_0 \wedge
  \bigwedge_{i=1}^{m-1} I^i_1) \wedge [\forall] ( \neg q_m \vee \ldots \vee
  \neg q_1 \vee \Box \neg w).
  $$

Finally, define
\[
\phi_T = Init \wedge R_1 \wedge R_2 \wedge R_3 \wedge R_4 \wedge R_5 \wedge R_6
\wedge R_7 \wedge R_8 \wedge R_9 \wedge Win \wedge F.
\]

\begin{theorem}
  \label{claim:tiling-reduction}
  Let $T$ be an instance of the two person corridor tiling game.
  Then, Eloise has a winning strategy in $T$ if, and only if, $\phi_T$
  is satisfiable in a Kripke model.
\end{theorem}

\begin{proof}
  Similar to the proof of Theorem 6.52 from~\cite{Blackburn}.
\end{proof}

Next, given an instance $T$ of the two person corridor tiling game, we construct a 
quantifier-free first-order $\sig^{bdo}_=$-formula $\vp_T$ (based on $\phi_T$) such that $\vp_T$
is satisfiable in a $bdo$ if, and only if, $\phi_T$ is satisfiable
in a Kripke model.
The formula $\vp_T$ is a conjunction of two parts: (1) $Tr(\phi_T)$,
the `translation' of $\phi_T$ and (2) $\zeta_T$, the conjunction of
identities that `simulate' the occurrences of $\neg$ and $\Box$ in
$\phi_T$ (which is needed since the language $\sig^{bdo}_=$
contains neither $\neg$ nor $\Box$).
%We denote by $SF (\phi_T)$ the set of subformulas of
%$\phi_T$.

First, distribute $\neg$ in $\phi_T$ over $\vee$ and $\wedge$ and
eliminate double negations (call the resultant formula $\phi_T$ to avoid notational clutter).  
Then, $\neg$ in $\phi_T$ only applies to variables.
%From now on, we treat each propositional variable in
%$\phi_T$ as a first-order variable.
Next, perform the following steps:
\begin{enumerate}
\item For every variable $p$ in $\phi_T$, introduce a fresh variable
  $p'$ and add to $\zeta_T$ the identities $p \vee p' = 1$ and
  $p \wedge p' = 0$.
\item For every subformula of $\phi_T$ of the form $\Box p$, introduce
  a fresh variable $b(p)$ and add to $\zeta_T$ the identities
  $b(p) \vee \Diamond p' = 1$ and $b(p) \wedge \Diamond p' = 0$.  
\item For every subformula of $\phi_T$ of the form $\Box \neg p$,
  introduce a fresh variable $bn(p)$ and add to $\zeta_T$ the identities
  $bn(p) \vee \Diamond p = 1$ and $bn(p) \wedge \Diamond p = 0$.
\item For every subformula of $\phi_T$ of the form
  $\Box ( p_1 \vee \ldots \vee p_k )$, introduce a fresh variable
  $bd(p_1, \ldots, p_k)$ and add to $\zeta_T$ the identities
\begin{align*}  
  & bd(p_1, \ldots, p_k) \vee \Diamond (p'_1 \wedge \ldots \wedge p'_k) = 1 \;\;\mbox{and} \\
 & bd(p_1, \ldots, p_k) \wedge \Diamond (p'_1 \wedge \ldots \wedge p'_k) = 0.
 \end{align*}
\item For every subformula of $\phi_T$ of the form
  $\Box \bigwedge_{i = 1}^{k} \bigvee_{j = 1}^{l} p_{i,j}$, introduce
  a fresh variable $bcd(p_{1,1}, \dots, p_{k,l})$, and add to $\zeta_T$ the identities
\begin{align*}
  & \mbox{$bcd(p_{1,1}, \dots, p_{k,l}) \vee \Diamond \bigvee_{i = 1}^{k} \bigwedge_{j = 1}^{l} p'_{i,j} = 1$}
  \;\; \mbox{and} \\
  & \mbox{$bcd(p_{1,1}, \dots, p_{k,l}) \wedge \Diamond \bigvee_{i = 1}^{k} \bigwedge_{j = 1}^{l} p'_{i,j} = 0$}.
  \end{align*}
\end{enumerate}

Construct a formula $Tr(\phi_T)$ as follows.  
Put $\phi_T$ into the form
$$\hat{\phi}_T = \chi \wedge [\forall] \psi_1 \wedge \ldots \wedge [\forall] \psi_m.$$ 
This can be done by bringing together all the conjucts of $\phi_T$
that do not have occurrences of $[\forall]$ into a single conjunction $\chi$.
Notice that none of $\chi, \psi_1, \ldots, \psi_m$ contain
occurrences of $[\forall]$ since $[\forall]$ never occurs within the
scope of a modal connective.
For every formula in $\{\chi, \psi_1, \ldots, \psi_m\}$, recursively define the
translation $\cdot^\ast$ of $\sigma^{m\ell u}$-formulas into $\sig^{bdo}_=$-terms, as follows:
\begin{center}
  \begin{tabular}{lll}
    $p^\ast$ & = & $p,$ $\;\mbox{where $p$ is a variable outside of the scope of $\neg$ and $\Box$}$; \\
    $(\neg p)^\ast$ & = & $p',$ $\;\mbox{where $p$ is a variable outside of the scope of $\Box$}$; \\
    $(\phi_1 \wedge \phi_2)^\ast$ & = & $(\phi_1)^\ast \wedge (\phi_2)^\ast$, 
    		$\mbox{where $\phi_1 \wedge \phi_2$ is outside of the scope of $\Box$}$; \\
    $(\phi_1 \vee \phi_2)^\ast$ & = & $(\phi_1)^\ast \vee (\phi_2)^\ast$, $\mbox{where $\phi_1 \vee \phi_2$ 
               is outside of the scope of $\Box$}$; \\
    $(\Diamond \phi)^\ast$ & = & $\Diamond ( \phi)^\ast$;  \\
    $(\Box p)^\ast$ & = & $b(p)$;  \\
    $(\Box \neg p)^\ast$ & = & $bn(p)$; 
\end{tabular}
\end{center}
\begin{tabular}{lll}
    $(\Box ( p_1 \vee \ldots \vee p_k ))^\ast$ & = & $bd(p_1, \ldots, p_k)$;  \\
    $(\Box \bigwedge_{i = 1}^{k} \bigvee_{j = 1}^{l} p_{i,j})^\ast$ & = & $bcd(p_{1,1}, \dots, p_{k,l})$. 
\end{tabular}

Now, let $Tr (\phi_T)$ be the following quantifier-free
$\sig^{bdo}_=$-formula:
$$
 (\mbox{\bf not} (\chi^\ast = 0)) \;\mbox{\bf and}\; \psi_1^\ast = 1 \;\mbox{\bf and}\;
\ldots \;\mbox{\bf and}\; \psi_m^\ast = 1.
$$
Finally, let $\vp_T$ be:
$ \zeta_T \;\mbox{\bf and}\; Tr(\phi_T)$.

\begin{lemma}
  Let $\phi_T$ and $\varphi_T$ be formulas constructed, as above, from
  an instance $T$ of the two person corridor tiling game.  Then,
  $\vp_T$ is satisfiable in a $bdo$ if, and only if, $\phi_T$ is
  satisfiable in a Kripke model.
\end{lemma}

\begin{proof}
  Assume that $\vp_T$ is true in some $bdo$
  ${\mathbb{A}} = \langle A, \wedge, \vee, \Diamond, 0, 1, \leqslant
  \rangle$
  under some valuation $v$.  Recall that the $bdo$-frame
  $\mathfrak{F}_{\mathbb{A}}$ associated with ${\mathbb{A}}$ is the
  structure \linebreak $\lb \mathcal P, \sub, \mathcal R \rb$, where
  $\mathcal P$ is the set of prime filters of ${\mathbb{A}}$ and
  $\mathcal R(F,G)$ iff $a \in G$ implies $\Diamond a \in F$, for all
  $a \in A$.  If we omit the order relation, then
  $\lb \mathcal P, \mathcal R \rb$ is a Kripke frame and
  $\mathfrak{M} = \langle \mathcal P, \mathcal R, V \rangle$, where
  $V(p) = \{ F \in \mathcal P \mid v(p) \in F \}$ for each variable
  $p$, is a Kripke model.  We show that
  $\mathfrak{M}, F_0 \models \phi_T$ holds for some
  $F_0 \in \mathcal P$.

  As $v(\chi^\ast) \ne 0$, there exists $F_0 \in \mathcal P$
  containing all of $v(e)$, $v(p_1)$, $v(c_0 (T_0))$,
  $v(c_1 (T_{I_1})), \dots$, $v(c_n (T_{I_n}))$, $v(c_{n+1} (T_0))$,
  $v(w)$, $v(q'_1), \dots, v(q'_m)$.  Therefore, \linebreak
  $\mathfrak{M}, F_0 \models \chi$.  It remains to show that, for
  every subformula of $\hat{\phi}_T$ of the form $[\forall] \psi$, the
  formula $\psi$ is true in $\mathfrak{M}$ at every
  $F \in \mathcal P$.  We consider one such $\psi$ as an example, say
  $\psi_0 = \neg w \vee ( c_1( T_{s+1}) \vee (e \wedge \Diamond w ) )
  \vee (\neg e \wedge \Diamond \top \wedge \Box w ) $,
  which is a subformula of $Win$.  Let $F \in \mathcal P$; we need to
  show that $\mathfrak{M}, F \models \psi_0$.  By assumption, all of
  the following are true in ${\mathbb{A}}$ under $v$ (we henceforth
  omit $v$ for ease of notation):
\begin{align*}
  w' \vee c_1( T_{s+1}) \vee (e \wedge \Diamond w ) \vee ( e' \wedge
  \Diamond \top \wedge b(w)) = 1; \\
    e \vee e' = 1; \quad w \vee w' = 1; \quad b(w) \vee \Diamond w' = 1; \\
    e \wedge e' = 0; \quad w \wedge w' = 0; \quad b(w) \wedge \Diamond w' = 0.
\end{align*}
  Since $1 \in F$, either $w'$ or $c_1( T_{s+1})$ or
  $e \wedge \Diamond w$ or $e' \wedge \Diamond \top \wedge b(w)$
  belongs to $F$.  
  Assume that $w' \in F$.
  As $w \wedge w' = 0$, we
  have $w \notin F$, i.e., $F \notin V(w)$, hence
  $\mathfrak{M}, F \models \neg w$, and so
  $\mathfrak{M}, {F} \models {\psi_0}$.
  It is straightforward to show
  that $\mathfrak{M}, F \models \psi_0$ when $c_1( T_{s+1}) \in F$ and
  when $e \wedge \Diamond w \in F$.
  Assume that
  $e' \wedge \Diamond \top \wedge b(w) \in F$; thus, $e' \in F$,
  $\Diamond \top \in F$, and $b(w) \in F$.  
  As $e \wedge e' = 0$, we
  have $e \notin F$, hence $\mathfrak{M}, F \models \neg e$.  As
  $b(w) \wedge \Diamond w' = 0$, we have $\Diamond w' \notin F$.
  Suppose that $R(F,G)$ for some $G \in P$; then,
  $w' \notin G$, and thus $w \in G$, as $w \vee w' = 1$; therefore,
  $\mathfrak{M}, F \models \Box w$, and so $\mathfrak{M}, F \models \psi_0$.

  Assume, on the other hand, that $\mathfrak{M}, w_0 \models \phi_T$
  for some Kripke model \linebreak
  $\mathfrak{M} = \langle W, R, V \rangle$ and some $w_0 \in W$.  Let
  $\leqslant$ be the trivial partial order on $W$, that is
  $w \leqslant z$ if, and only if, $w = z$.  Then,
  $\mathfrak{F} = \langle W, \leqslant, R \rangle$ is a $bdo$-frame
  and so we can construct its associated $bdo$
  ${\mathbb{A}}_\mathfrak{F} = \langle \mathfrak{U}(W), \cup, \cap,
  \Diamond, \varnothing, W, \sub \rangle$
  as in Section~\ref{sec:rdl}.  It remains to define a valuation $v$
  that will satisfy $\vp_T$.  For the variables of $\phi_T$, let
  $v(p) = \{ w \in W \mid w \in V(p) \}$.  For the variables not
  occuring in $\phi_T$, but occurring in $\vp_T$, the valuation $v$ is
  defined as follows:
  \begin{align*}
    v(p') & = \{ w \in W \mid w \notin V(p) \}; \\
    v(bp)& = \{ w \in W \mid R(w,v) \mbox{ implies } v \in V(p) \}; \\
    v(bn(p)) & = \{ w \in W \mid R(w,v) \mbox{ implies } v \notin V(p)
                   \}; \\
    v(bd(p_1, \ldots, p_k)) & = \{ w \in W \mid R(w,v) \mbox{ implies }
                                  v \in V(p_1) \cup \ldots \cup
                                  V(p_k) \}; \\
    v(bcd(p_{1,1}, \dots, p_{k,l})) & = \{ w \in W \mid R(w,v) \mbox{ implies }
                      v \in \mbox{$\bigcap_{i=1}^{k} \bigcup_{j=1}^{l} V(p_{i,j})$}  \}.
  \end{align*}
It is then straightforward to check that $\vp_T$ is true in ${\mathbb{A}}_\mathfrak{F}$ under $v$.
\end{proof}

Observe that constructing the formula $\vp_T$ from an instance $T$ of the two person corridor tiling game
can be done in polynomial time.
The above results, therefore, give us a polynomial time reduction of the 
two person corridor tiling problem to satisfiability in $BDO$.
Thus, we have the following.

\begin{theorem}
  \label{cl:BDOs-lower-bound}
Satisfiability in $BDO$ is {\em \textsf{EXPTIME}}-hard.
\end{theorem}

We now apply Theorem~\ref{cl:BDOs-lower-bound} to obtain similar results for the classes $BDBO$ and $BRDG$.

\begin{corollary}
Satisfiability in $BDBO$ and in $BRDG$ are both {\em \textsf{EXPTIME}}-hard.
\end{corollary}
\begin{proof}
Let $\vp_T$ be a quantifier-free first-order $\sig^{bdo}_=$-formula constructed from an instance $T$ of the 
two-person corridor tiling game, as above.
Recall that $\vp_T$ is satisfiable in a $bdo$ if, and only if, 
Eloise has a winning strategy in the game instance $T$.
In polynomial time, we may construct the formula $\vp^{\circ}_T$ from $\varphi_T$ as in Section~\ref{sec:dol-complexity}.
Recall, from Lemma~\ref{vp* lem}, that 
$\vp_T$ is satisfiable in a $bdo$ if, and only if, $\vp^{\circ}_T$ is satisfiable in a $bdbo$.
In addition, by Lemma~\ref{vp dlbo lem}, $\vp^{\circ}_T$ is satisfiable in a $bdbo$ if, and only if,
$\vp^{\circ}_T$ is satisfiable in a $brdg$.
The result now follows from the \textsf{EXPTIME}-hardness of the tiling problem.
\end{proof}

Recall that, for a class $\mathcal K$ and universal sentence $\Phi$, we have
$\mathcal K \models \Phi$ iff $\mbox{\bf not\,}\varphi$ is not
satisfiable in $\mathcal K$, where $\varphi$ is the quantifier-free part of $\Phi$.  
This yields the connection between the complexity of the universal theory of $\mathcal K$
and the complexity of satisfiability in $\mathcal K$.
Noting also that the complement of an  \textsf{EXPTIME}-complete problem is also \textsf{EXPTIME}-complete, we obtain 
our main result.

\begin{theorem}
  \label{thr:BDOs-complexity}
Satisfiability in each of the classes $BDO$, $BDBO$ and $BRDG$ is {\em \textsf{EXPTIME}}-complete.
The universal theory of  each of the classes $BDO$, $BDBO$ and $BRDG$ is
{\em \textsf{EXPTIME}}-complete.
\end{theorem}

Recall that a quasi-identity (or quasi-equation) is a universal sentence
of the form $(\forall x_1, \dots, x_n)(s_1 = t_1 \;\mbox{\bf and}\; \dots \;\mbox{\bf and}\; s_n = t_n \Rightarrow u = v)$,
and that the quasi-equational theory of a class $\mathcal K$ is the set
of all quasi-identities $\Phi$ such that $\mathcal K \models \Phi$.
Consider a quantifier-free formula $\varphi_T$ constructed from an instance $T$ of the two person corridor tiling game, as above.
Observe that if the existential quantifiers are added to $\varphi_T$, then the resulting sentence is the negation of a quasi-identity.
Thus, the above proofs establish the lower bound for the quasi-equational theory of 
each of the classes $BDO$, $BDBO$ and $BRDG$.
Thus, we have the following result.

\begin{corollary} \label{qe cor}
The quasi-equational theory of each of the classes $BDO$, $BDBO$ and $BRDG$ is {\em \textsf{EXPTIME}}-complete.
\end{corollary}

It follows from Corollary~\ref{qe cor} that the consequence relation for Distributive Nonassociative Full
Lambek Calculus with Bounds is \textsf{EXPTIME}-complete.

Since \textsf{P} $\neq$ \textsf{EXPTIME} and \textsf{P} is conventionally considered to be the class of `tractable' problems,
the above results show that the universal and quasi-equational theories for each of the classes $BDO$, $BDBO$ and $BRDG$ are `intractable'.

\section{Special classes of $brdg$'s}
\label{sec:special-classes-rdl}

We now consider complexity of universal theories of special classes of
$brdg$'s that correspond to standard extensions of Distributive Nonassociative Full Lambek Calculus with Bounds.
In particular, we consider the following properties of $\circ$:
\begin{itemize}
  \item[(P1)] {\em commutative}:\; $x \circ y = y \circ x$
  \item[(P2)] {\em decreasing}:\; $x \circ y \leqslant x$ and $x \circ y \leqslant y$
  \item[(P3)] {\em square-increasing}:\; $x \leqslant x \circ x$
  \item[(P4)] {\em unital}:\;  there exists $e$ such that $x \circ e = e \circ x = x$. 
  \end{itemize}

There is a correspondence between the properties of $\circ$
listed above and properties of the relation $R$ on $brdg$-frames $\lb P, \leqslant, R \rb$.
The frame conditions corresponding to (P1--P4) are as follows (see, e.g., \cite{DH}):
\begin{itemize}
\item[(R1)]  $(\forall x, y, z \in P)\, ( R(x, y, z) \Rig R(y, x, z))$
\item[(R2)]  $(\forall x, y , z \in P)\, ( R(x, y ,z) \Rig (x \leqslant z \;\mbox{\bf and}\; y \leqslant z))$
\item[(R3)]  $(\forall x \in P)\, R(x,x,x)$
\item[(R4)] There exists a set $E  \subseteq P$ such that 
\begin{align*}
& (\forall x, y, z \in P) (R(x, y, z) \;\mbox{\bf and}\; y \in E \Rig x \leqslant z) \\
& (\forall x, y, z \in P) (R(x, y, z) \;\mbox{\bf and}\; x \in E \Rig y \leqslant z) \\
& (\forall x \in P)[(\exists y \in E) R(x, y, x) \;\mbox{\bf and}\; (\exists z \in E) R(z, x, x)].
\end{align*}
\end{itemize}
If ${\mathbb{A}}$ is a $brdg$ satisfying any subset of the properties  (P1--P4), then the associated
$brdg$-frame $\mathfrak{F}_{\mathbb{A}}$ satisfies the corresponding properties in (R1--R4).
In addition, if $\mathfrak{F}$ is a $brdg$-frame satisfying any subset of the properties (R1--R4),
then the associated $brdg$ ${\mathbb{A}}_\mathfrak{F}$ satisfies the corresponding properties in (P1--P4).
We also note that in the case of (P4), the upward closure of the set $E$ is the identity element of ${\mathbb{A}}_\mathfrak{F}$.

We seek to characterize partial algebras for the special classes of $brdg$'s 
considered above.
We describe the changes required in the statement and proof of Theorem~\ref{partial dlrg thm} in each case.

In the case that ${\mathbb{A}}$ is a commutative $brdg$, we replace $R^{\mathbb{B}}$ in the statement
of Theorem~\ref{partial dlrg thm} by the relation $R^{\mathbb{B}}_1$ defined by:
\[
     R^{\mathbb{B}}_1(f,g,h)  \leftrightharpoons [R^{\mathbb{B}}(f,g,h) \;\mbox{\bf and}\; R^{\mathbb{B}}(g,f,h)],
\]
where $R^{\mathbb{B}}$ is as before, that is,
 \[
   \begin{array}{rcl}
     R^{\mathbb{B}} ( f, g, h ) &  \leftrightharpoons & (\forall\, (a,b) \in \dom(\circ^{\mathbb{B}}))
                                                     ( a \in f \;\mbox{\bf and}\;
                                                     b \in g \Rig a \circ^{\mathbb{B}} b \in h)
                                                     \;\mbox{\bf and}\;
     \\ & & (\forall\, (a,b) \in \dom(\lres^{\mathbb{B}} )) ( a \in f \;\mbox{\bf and}\; a \lres^{\mathbb{B}} b \in
            g \Rig b \in h) \;\mbox{\bf and}\;
     \\ & & (\forall\, (a,b) \in \dom(\rres^{\mathbb{B}} )) ( b \rres^{\mathbb{B}} a \in f \;\mbox{\bf and}\; a  \in
            g \Rig b \in h).
   \end{array}
\]
For the proof, suppose that ${\mathbb{B}}$ is a partial substructure of a commutative $brdg$ ${\mathbb{A}}$.
As in the proof of Theorem~\ref{partial dlrg thm}, define the set
 \[
  \mathcal{F} := \{ F \cap B \mid F \mbox{ is a prime filter of } {\mathbb{A}} \}.
\]
Then the $brdg$-frame $\langle \mathcal P, \sub, \mathcal R \rangle$ associated with ${\mathbb{A}}$ satisfies (R1)
and it is straightforward to check that \eqref{R and R} holds for the relation $R^{\mathbb{B}}_1$.
It then follows, as in the proof of Theorem~\ref{partial dlrg thm}, that
(D), (M$_\circ$), (M$_\lres$) and (M$_\rres$) hold, with $R^{\mathbb{B}}_1$ replacing $R^{\mathbb{B}}$.
Conversely, if we start with a partial $\sig^{brdg}$-structure ${\mathbb{B}}$ satisfying all the requirements of
Theorem~\ref{partial dlrg thm} with the relation $R^{\mathbb{B}}_1$ replacing $R^{\mathbb{B}}$, then 
the structure $\mathfrak{F} = \lb \mathcal F, \sub, R^{\mathbb{B}}_1 \rb$ satisfies (\ref{FR1}--\ref{FR3}), hence it is a
$brdg$-frame and it clearly satisfies (R1). 
Thus, the $brdg$ ${\mathbb{A}}_\mathfrak{F}$ constructed from $\mathfrak{F}$, and into which ${\mathbb{B}}$ embeds, is commutative.

Commutative $brdg$'s satisfy $x \lres y = y \rres x$ and so a single binary operation $\rig$ is usually used
in place of $\lres$ and $\rres$, replacing $x \lres y$ and $y \rres x$ by $x \rig y$.
With this change the description of partial substructures of commutative $brdg$s can be 
written in a simpler form.
We retain our notation, however, for uniformity of presentation.

For decreasing $brdg$'s the required change to Theorem~\ref{partial dlrg thm} is that $R^{\mathbb{B}}$ is replaced by 
the relation $R^{\mathbb{B}}_2$ defined by:
\[
   R^{\mathbb{B}}_2(f,g,h) \leftrightharpoons [R^{\mathbb{B}}(f,g,h) \;\mbox{\bf and}\; f \sub h \;\mbox{\bf and}\; g \sub h].
\]
If ${\mathbb{B}}$ is a partial substructure of a decreasing $brdg$ ${\mathbb{A}}$, then, in this case, the $brdg$-frame 
$\langle \mathcal P, \sub, \mathcal R \rangle$ associated with ${\mathbb{A}}$ satisfies (R2).
It then follows that  \eqref{R and R} holds for the relation $R^{\mathbb{B}}_2$ and that
(D), (M$_\circ$), (M$_\lres$) and (M$_\rres$) hold.
Conversely, if we start with a partial $\sig^{brdg}$-structure ${\mathbb{B}}$ satisfying all the requirements of
Theorem~\ref{partial dlrg thm} with the relation $R^{\mathbb{B}}_2$ replacing $R^{\mathbb{B}}$, then 
the structure $\mathfrak{F} = \lb \mathcal F, \sub, R^{\mathbb{B}}_2 \rb$ satisfies (\ref{FR1}--\ref{FR3}) and (R2). 
Thus, the $brdg$ ${\mathbb{A}}_\mathfrak{F}$, constructed from $\mathfrak{F}$ and into which ${\mathbb{B}}$ embeds, is decreasing.

For square-increasing $brdg$'s the required change to Theorem~\ref{partial dlrg thm} is to the definition
of a prime filter of the partial structure ${\mathbb{B}}$.
A set $f \sub B$ is now called a prime filter of ${\mathbb{B}}$ if it satisfies (\ref{PF1}--\ref{PF4}), as well as the following:
\begin{align}
& \mbox{$(\forall (a,b) \in \dom(\circ^{{\mathbb{B}}})) (a \in f \;\mbox{\bf and}\; b \in f \Rig a \circ^{{\mathbb{B}}} b \in f)$}, \label{PFnew1} \\
& \mbox{$(\forall (a,b) \in \dom(\lres^{{\mathbb{B}}})) (a \in f \;\mbox{\bf and}\; a \lres^{\mathbb{B}}b \in f \Rig b \in f)$}, \label{PFnew2} \\
& \mbox{$(\forall (a,b) \in \dom(\rres^{{\mathbb{B}}})) (a \rres^{\mathbb{B}} b \in f \;\mbox{\bf and}\; b \in f \Rig a \in f)$.} \label{PFnew3}
\end{align}
If ${\mathbb{B}}$ is a partial substructure of a square-increasing $brdg$ ${\mathbb{A}}$, then the proofs that
(D), (M$_\circ$), (M$_\lres$) and (M$_\rres$) hold (with respect to the relation $R^{\mathbb{B}}$) are as before.
We need only check that each $F \cap B \in \mathcal F$ satisfies (\ref{PFnew1}--\ref{PFnew3}), which
follows from the fact that any prime filter $F$ of a square-increasing $brdg$ is closed under $\circ$.
Conversely, if we start with a partial $\sig^{brdg}$-structure ${\mathbb{B}}$ satisfying all the requirements of
Theorem~\ref{partial dlrg thm} with the additional requirements (\ref{PFnew1}--\ref{PFnew3}) on elements of $\mathcal F$, then 
it is clear that the structure $\mathfrak{F} = \lb \mathcal F, \sub, R^{\mathbb{B}} \rb$ is a $brdg$-frame 
for which $R^{\mathbb{B}}(f,f,f)$ holds for each $f \in \mathcal F$, i.e., (R3) holds. 
Thus, the $brdg$ ${\mathbb{A}}_\mathfrak{F}$, constructed from $\mathfrak{F}$ and into which ${\mathbb{B}}$ embeds, is 
square-increasing.

In the case of unital $brdg$'s, we require that the identity element $e$ be included in the language
so that, by our convention, the identity element is an element of every partial substructure.
Thus, we define the language $\sigma^{brdge}$ as $\sigma^{brdg}$ augmented with a constant symbol $e$
and by a $brdge$ we mean a $\sigma^{brdge}$-algebra whose $\sigma^{brdg}$-reduct is a $brdg$ and that satisfies
$x \circ e = x = e \circ x$. 
To characterize partial $brdge$'s, in the statement of Theorem~\ref{partial dlrg thm}, 
we replace $R^{\mathbb{B}}$ by the relation $R^{\mathbb{B}}_4$ defined by
\begin{align*}
    R^{\mathbb{B}}_4(f,g,h) \Leftrightarrow\; & [R^{\mathbb{B}}(f,g,h) \;\mbox{\bf and} \\
    & (\forall\, a \in B) (a \in f \;\mbox{\bf and}\; e^{\mathbb{B}} \in g \Rig a \in h) \;\mbox{\bf and}\; \\
    & (\forall\, a \in B) (e^{\mathbb{B}} \in f \;\mbox{\bf and}\; a \in g \Rig a \in h)].
\end{align*}
In addition, we require that the following condition holds:
\begin{itemize}
\item[(M$_e$)] $(\forall f \in \mathcal F)[(\exists g \in \mathcal F)(e^{\mathbb{B}} \in g \;\mbox{\bf and}\; R_4^{\mathbb{B}}(f,g,f))$ \\[0,2cm]
$\mbox{}\hspace{3cm} \;\mbox{\bf and}\; (\exists h \in \mathcal F)(e^{\mathbb{B}} \in h \;\mbox{\bf and}\; R_4^{\mathbb{B}}(h, f ,f))]$.
\end{itemize}
For the proof, suppose that ${\mathbb{B}}$ is a partial substructure of a $brdge$ ${\mathbb{A}}$.
Then the $brdg$-frame $\langle \mathcal P, \sub, \mathcal R \rangle$ associated with ${\mathbb{A}}$ satisfies (R4).
Define the set $\mathcal{F}$ as above.
It is straightforward to check that \eqref{R and R} holds for the relation $R^{\mathbb{B}}_4$, from which it follows,
as in the proof of Theorem~\ref{partial dlrg thm}, that
(D), (M$_\circ$), (M$_\lres$) and (M$_\rres$) hold.
To see that (M$_e$) holds, we require the following result.

\begin{lemma} \cite{DH}  \label{id lem}
Let ${\mathbb{A}}$ be a $brdge$.
For any prime filter $F$ of ${\mathbb{A}}$, there exists a prime filter $G$ of ${\mathbb{A}}$ with $e^{\mathbb{A}} \in G$
and $\mathcal R(F,G,F)$ and there exists a prime filter $H$ of ${\mathbb{A}}$ with $e^{\mathbb{A}} \in H$ and 
$\mathcal R(H,F,F)$.
\end{lemma}

Let $f \in \mathcal F$, so $f = F \cap B$ for some prime filter $F$ of ${\mathbb{A}}$.
By Lemma~\ref{id lem}, there exists a prime filter $G$ of ${\mathbb{A}}$ such that $e^{\mathbb{A}} \in G$ and $\mathcal R(F,G,F)$.
Then, for $g := G \cap B$, we have $e^{\mathbb{B}} = e^{\mathbb{A}}  \in g$, while $R_4^{\mathbb{B}}(f,g,f)$
follows from \eqref{R and R}.
The required $h$ is obtained similarly. 

Conversely, suppose ${\mathbb{B}}$ is a partial $\sig^{brdge}$-structure satisfying the requirements above.
Then $\lb \mathcal F, \sub, R_4^{\mathbb{B}} \rb$ is a $brdg$-frame; we show that it satisfies (R4)
so that the constructed $brdg$ ${\mathbb{A}}_{\mathfrak F}$ is a $brdge$.
Set $\mathcal E := \{ g \in \mathcal F : e^{\mathbb{B}} \in g \}$.
Suppose $f,g,h \in \mathcal F$ are such that $R_4^{\mathbb{B}}(f,g,h)$ and $g \in \mathcal E$, and let $a \in f$. 
Since $e^{\mathbb{B}} \in g$, we have $a \in h$, by definition of $R_4^{\mathbb{B}}$. 
Thus, $f \sub h$.
Similarly, if $R_4^{\mathbb{B}}(f,g,h)$ and $f \in \mathcal E$, then $g \sub h$.
The third condition in (R4) is immediate from (M$_e$).
Lastly, the identity element of ${\mathbb{A}}_{\mathfrak F}$ is $\mathcal E$ and
the embedding $\mu$, as defined in Theorem~\ref{partial dlrg thm}, satisfies 
$\mu(e^{\mathbb{B}}) = \{ f \in \mathcal F : e^{\mathbb{B}} \in f \} = \mathcal E$.

We note that the above adaptations to the characterizations of partial substructures of $brdg$'s satisfying properties
(P1--P4) can be combined in obvious ways to provide characterizations of partial substructures of
$brdg$'s satisfying any subset of the properties (P1--P4).
In addition, the \textsf{EXPTIME}-algorithm for satisfiability in $BRDG$ described in Section~\ref{sec:dol-complexity} 
can be easily adapted to cater for these characterizations.
Indeed, for (P1) and (P2) we need only use a different version of $R^{\mathbb{B}}$, and for (P3)
we need only check that the prime filters satisfy the additional requirements.
For (P4), we require relation $R_4^{\mathbb{B}}$ and we
must additionally check condition (M$_e$) for each $f$ during each traversal of the current set $\mathcal F$.
If there does not exist $g \in \mathcal F$ with $e^{\mathbb{B}} \in g$ and $R^{\mathbb{B}}(f,g,f)$, then
we remove $f$ from $\mathcal F$; similarly, for $h$.
This does not change the exponential running time of the algorithm.
Thus, by the methods of Section~\ref{sec:dol-complexity}, we get an upper bound of exponential time
for the complexity of satisfiability in the class of $brdg$'s with any combination of the additional properties (P1--P4).
As regards the lower bound, we can only infer the exponential time lower bound for $brdg$'s with
the commutative property, which follows from the fact that the operation
$\circ$ defined in Lemma~\ref{lem:dlo-to-rdl} is commutative.
Thus, we have the following result.

\begin{theorem}  \label{special classes thm}
For any subset $\mathcal Q$ of the properties {\em (P1--P4)}, let 
$BRDG^{\mathcal Q}$ be the class of all $brdg$'s satisfying the properties in $\mathcal Q$.
Then, satisfiability in $BRDG^{\mathcal Q}$ and the universal theory of $BRDG^{\mathcal Q}$ are both in {\em \textsf{EXPTIME}}.
In the case where $\mathcal Q$ contains only {\em (P1)}, that is, for commutative $brdg$'s,
satisfiability in $BRDG^{\mathcal Q}$ and the universal theory of $BRDG^{\mathcal Q}$ are both {\em \textsf{EXPTIME}}-complete.
\end{theorem}

\section{Discussion}

We conclude by mentioning some unresolved problems closely related to the results presented in this paper.

Theorem~\ref{special classes thm} provides an upper bound for the complexity of the universal theory of
each class $BRDG^{\mathcal Q}$, where $\mathcal Q$ is a subset of (P1--P4).
Tight complexity bounds for these classes are not known, with three exceptions.
One is the case where $\mathcal Q$ is empty, that is for $BRDG$; 
another is where $\mathcal Q$ contains just the commutative property (P1), as discussed in the previous section.
The other is the case where $\mathcal Q$ contains all properties (P1--P4), or even just (P2--P3), from which (P1) and (P4) follow;
in this case, $BRDG^{\mathcal Q}$ is the variety of Heyting algebras, whose universal theory is known to be \textsf{PSPACE}-complete
(see \cite{VA13}).

If we remove the bounds $0$ and $1$ from the language of $brdg$'s, we obtain the class of `residuated distributive
lattice-ordered groupoids', or $rdg$'s for short.
Since any $rdg$ can be embedded into a $brdg$ (in a straightforward way),
the \textsf{EXPTIME} upper bound for the universal theory of $BRDG$ holds for 
the class of $rdg$'s as well; however, the proof of the \textsf{EXPTIME}
lower bound presented in this paper requires the bounds, so the tight complexity of the universal theory of this class 
remains unknown.

Lastly, the complexity of the equational theories of $BRDG$, $BDO$ and $BDBO$, to the best of our knowledge, has not been settled.

\section{Acknowledgements}

We are indebted to the anonymous reviewers for
their suggestions that have helped to improve the paper.

\bibliographystyle{plain}

\end{document}